\documentclass{amsart}
\usepackage{amssymb}
\usepackage[pdfencoding=auto]{hyperref}
\usepackage{mathtools}
\usepackage{comment}
\usepackage{enumitem}
\usepackage{graphicx}
\usepackage{ifthen}
\usepackage[svgnames]{xcolor}
\usepackage{scalerel}
\usepackage{mathtools}

\usepackage[capitalise]{cleveref}

\setlist{}
\setenumerate{leftmargin=*,labelindent=0\parindent}
\setitemize{leftmargin=\parindent}

\usepackage{tikz}
\usepackage{tikz-cd}

\newtheorem{thm}{Theorem}[subsection]
\newtheorem{lem}[thm]{Lemma}
\newtheorem{prop}[thm]{Proposition}

\newtheorem{cor}[thm]{Corollary}

\theoremstyle{definition}
\newtheorem{defn}[thm]{Definition}
\newtheorem{ex}[thm]{Example}
\newtheorem{rmk}[thm]{Remark}

\newtheorem{ntn}[thm]{Notation}

\newtheorem{const}[thm]{Construction}

\newtheorem{cvn}[thm]{Convention}

\crefname{lem}{Lemma}{Lemmas}
\crefname{thm}{Theorem}{Theorems}
\crefname{defn}{Definition}{Definitions}
\crefname{notn}{Notation}{Notations}
\crefname{const}{Construction}{Constructions}
\crefname{prop}{Proposition}{Propositions}
\crefname{rmk}{Remark}{Remarks}
\crefname{cor}{Corollary}{Corollaries}
\crefname{ex}{Example}{Examples}
\crefname{nex}{Non-Example}{Non-Examples}

\makeatletter
\let\c@equation\c@thm
\makeatother
\numberwithin{equation}{subsection}

\newcommand{\op}{\textup{op}}
\newcommand{\id}{\textup{id}}
\newcommand{\ob}{\textup{ob}}
\newcommand{\cod}{\textup{cod}}
\newcommand{\dom}{\textup{dom}}
\newcommand{\ho}{\textup{ho}}

\makeatletter
\def\@mathlower#1#2#3{\setbox0=\hbox{$\m@th#2#3$}\lower#1\ht0\box0}
\def\mathlower#1#2{\mathpalette{\@mathlower{#1}}{#2}}
\makeatother

\newcommand{\wto}{\xrightarrow{{\smash{\mathlower{0.8}{\sim}}}}}
\newcommand{\cwtoo}[2]{\begin{tikzcd}[cramped, sep=small,ampersand replacement=\&] #1 \rar[tail, "\sim" style={outer sep=-1pt}] \& #2 \end{tikzcd}}
\newcommand{\fwtoo}[2]{\begin{tikzcd}[cramped, sep=small,ampersand replacement=\&] #1 \rar[two heads, "\sim" style={pos=0.4,outer sep=-1pt}] \& #2 \end{tikzcd}}

\newcommand{\slicel}[2]{\vphantom{#2}^{{#1}/}\mkern-2mu{#2}}

\newcommand{\TTheta}{\mathbf{\Theta}}
\newcommand{\DDelta}{\mathbf{\Delta}}

\newcommand{\cat}[1]{\textup{\textsf{#1}}}
\newcommand{\Set}{\cat{Set}}
\newcommand{\Cat}{\cat{Cat}}
\newcommand{\sSet}{\cat{sSet}}
\newcommand{\sCat}{\cat{sCat}}
\newcommand{\twoCat}{2\text{-}\Cat}
\newcommand{\Gaunt}{\cat{Gaunt}}

\DeclareSymbolFont{bbold}{U}{bbold}{m}{n}
\DeclareSymbolFontAlphabet{\mathbbb}{bbold}
\newcommand{\catone}{\ensuremath{\mathbbb{1}}}
\newcommand{\cattwo}{\ensuremath{\mathbbb{2}}}

\newcommand{\catn}{\ensuremath{\mathbbb{n}}}

\newcommand{\mG}{{\mathfrak{G}}}
\newcommand{\mF}{{\mathfrak{F}}}
\newcommand{\mP}{{\mathfrak{P}}}

\newcommand{\peq}{\preccurlyeq}
\newcommand{\npeq}{\not\preccurlyeq}

\newcommand{\inp}{\textup{in}}
\newcommand{\out}{\textup{out}}
\newcommand{\pred}{\textup{pred}}
\newcommand{\suc}{\textup{succ}}

\newcommand{\cA}{\mathcal{A}}
\newcommand{\cB}{\mathcal{B}}
\newcommand{\cC}{\mathcal{C}}
\newcommand{\cD}{\mathcal{D}}
\newcommand{\cP}{{\mathcal{P}}}
\newcommand{\cQ}{\mathcal{Q}}

\newcommand{\cV}{\mathcal{V}}
\newcommand{\cW}{\mathcal{W}}

\newcommand{\CC}{\mathbb{C}}

\newcommand{\aamalg}[1]{\underset{#1}{{\amalg}}} 
\begin{document}

\title{An \texorpdfstring{$(\infty,2)$}{(infinity,2)}-categorical pasting theorem}

\author{Philip Hackney}
\address{Department of Mathematics,
University of Louisiana at Lafayette, LA, USA
}
\email{philip@phck.net} 

\author{Viktoriya Ozornova}
\address{Max Planck Institute for Mathematics, Bonn, Germany}
\email{viktoriya.ozornova@mpim-bonn.mpg.de}

\author{Emily Riehl}
\address{Department of Mathematics\\Johns Hopkins University \\ 
Baltimore, MD 
 \\ USA}
\email{eriehl@jhu.edu}

\author{Martina Rovelli}
\address{Department of Mathematics and Statistics, 
University of Massachusetts 
Amherst,
MA, 
USA
}
\email{rovelli@math.umass.edu}

\date{\today}

\begin{abstract}
    We show that any pasting diagram in any $(\infty,2)$-category has a homotopically unique composite. This is achieved by showing that the free 2-category generated by a pasting scheme is the homotopy colimit of its cells as an $(\infty,2)$-category. We prove this explicitly in the simplicial categories model and then explain how to deduce the model-independent statement from that calculation.
\end{abstract}

\thanks{This material is based upon work supported by the National Science Foundation under Grant No.\ DMS-1440140, which began while the authors were in residence at the Mathematical Sciences Research Institute in Berkeley, California, during the Spring 2020 semester. The authors learned a lot from fruitful discussions with the other members of the MSRI-based working group on $(\infty,2)$-categories. This work was supported by a grant from the Simons Foundation (\#850849, PH). The second author thankfully acknowledges the financial support by the DFG grant OZ 91/2-1 with the project nr.~442418934. The third author is also grateful for support from the NSF via DMS-1652600, from the ARO under MURI Grant W911NF-20-1-0082, from the AFOSR under award number FA9550-21-1-0009, and by the Johns Hopkins President's Frontier Award program. The fourth author is deeply appreciative of the Mathematical Sciences Institute at the Australian National University for their support during the pandemic year. We are also grateful to an anonymous referee, who caught several small gaps in proofs in the initial manuscript and prompted us to make additional clarifications.}

\maketitle

\setcounter{tocdepth}{2}
\tableofcontents 

\section{Introduction}
A compatibly-oriented collection of 2-cells in a 2-category or bicategory can be composed by \emph{pasting}, a notion first introduced by B\'{e}nabou \cite{Benabou-Bicategories}. 
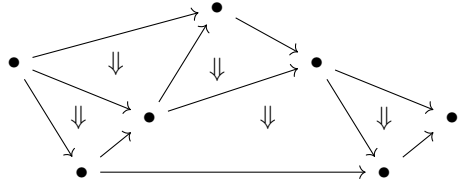
\[
\begin{tikzcd}[sep=small] & \arrow[ddr, phantom, "\displaystyle\Downarrow"]  & & \bullet \arrow[drr] \\
\bullet \arrow[ddr] \arrow[drr] \arrow[urrr] & & & & &  \bullet \arrow[drr] \arrow[ddr] & \arrow[dd, phantom, "\displaystyle\Downarrow"]  \\ \arrow[rr, phantom, "\displaystyle\Downarrow"] 
& & \bullet\arrow[uur] \arrow[urrr] \arrow[rrrr, phantom, "\displaystyle\Downarrow"] & \arrow[uu, phantom, "\displaystyle\Downarrow"]   & &~ & ~ & \bullet \\ & \bullet \arrow[ur] \arrow[rrrrr]  & & & & &  \bullet \arrow[ur]
\end{tikzcd}
\]
A celebrated theorem of Power \cite{PowerPasting} demonstrates that any 2-categorical pasting scheme defines a unique composite 2-cell in any 2-category. The essential uniqueness of pasting composites in a bicategory was proven by Verity \cite{Verity-Thesis} as a consequence of the bicategorical coherence theorem and later by a direct proof in a recent paper of Johnson and Yau \cite{JY}, also appearing in their book \cite{JYbook}.

The upshot of these results is a justification of the technique of pasting diagram chases in a 2-category, which had already been put to use, e.g., in \cite{KS}. For instance, to prove that any two left adjoints $f,f'$ to a given 1-cell $u \colon a \to b$ in a 2-category are isomorphic, one is led to consider the pasted composites of the units and counits of these adjunctions, displayed below-left, and their composite below-right: 
\[
\begin{tikzcd}[column sep=small] 
b \arrow[dr, "{f'}"'] \arrow[rr, equals] & \arrow[d, phantom, "\scriptstyle\Downarrow\eta'"] & b \arrow[dr, "f"]  &  & b \arrow[dr, "{f}"'] \arrow[rr, equals] & \arrow[d, phantom, "\scriptstyle\Downarrow\eta"] & b \arrow[dr, "f'"] & & & b \arrow[dr, "{f'}"'] \arrow[rr, equals] & \arrow[d, phantom, "\scriptstyle\Downarrow\eta'"] & b \arrow[dr, "f" description]   \arrow[rr, equals] & \arrow[d, phantom, "\scriptstyle\Downarrow\eta"] & b \arrow[dr, "f'"]  \\ & a \arrow[ur, "u" description] \arrow[rr, equals] & \arrow[u, phantom, "\scriptstyle\Downarrow\epsilon"] & a &  & a \arrow[ur, "u" description] \arrow[rr, equals] & \arrow[u, phantom, "\scriptstyle\Downarrow\epsilon'"] & a & & & a \arrow[ur, "u" description] \arrow[rr, equals] & \arrow[u, phantom, "\scriptstyle\Downarrow\epsilon"] & a  \arrow[ur, "u" description] \arrow[rr, equals] & \arrow[u, phantom, "\scriptstyle\Downarrow\epsilon'"] & a.
\end{tikzcd}
\]
By uniqueness of pasting composition, the right-hand pasted composite can be computed by first applying the triangle equality of $f \dashv u$ to reduce the composite $u\epsilon \cdot \eta u$ to the identity of $f$, and then applying the triangle equality from $f' \dashv u$ to reduce the composite $\epsilon' f \cdot f'\eta$ to an identity. A similar computation proves that the 2-cells displayed above-left are also inverses when composed the other way around, exhibiting the desired isomorphism $f \cong f'$.

The aim in this paper is to prove the corresponding pasting theorem for $(\infty,2)$-cat\-e\-gories, where we must alter the meaning of both ``unique'' and ``composite.'' In a weak infinite-dimensional category, composition is no longer a function but a generalized relation witnessed by higher dimensional cells, whose dimension is determined by the number of cells being composed. More exactly a composite of a collection of cells is witnessed by a homotopy coherent diagram, containing subdiagrams that witness binary composition relations, and ternary composition relations and so on. The composite cells by themselves are not unique but they are all equivalent, connected by higher cells that are weakly invertible.  In fact, more is true: our main theorem proves that a suitably-defined space whose points are homotopy coherent diagrams witnessing the composites of a given pasting scheme is contractible.

{
\renewcommand{\thethm}{\ref{cor:contractible-composites}}
\begin{cor}
The space of composites of any pasting diagram in any $(\infty,2)$-cat\-e\-gory is contractible.
\end{cor}
\addtocounter{thm}{-1}
}

\subsection{Homotopically unique composition in an \texorpdfstring{$(\infty,1)$}{(infinity,1)}-category}

Before describing the proof of our result, it is instructive to recall the proof of the homotopical uniqueness of composition in an $(\infty,1)$-category. The specifications of each particular model of $(\infty,1)$-categories build in the homotopical uniqueness of certain composites of 1-cells. For instance, quasi-categories are defined to be simplicial sets $X$ that are local for the inner horn inclusions $\cwtoo{\Lambda^n_k}{\Delta^n}$ meaning that $\fwtoo{X^{\Delta^n}}{X^{\Lambda^n_k}}$ is a trivial fibration. It follows that the inclusion $\Gamma^n\subset\Delta^n$ of the ``spine'' formed by the longest composable path of edges is a weak equivalence, and $\fwtoo{X^{\Delta^n}}{X^{\Gamma^n}}$ is a trivial fibration, so in particular the fibers are contractible Kan complexes.
Thus, in a quasi-category any sequence of $n$-composable morphisms---the 1-dimensional analogue of a pasting diagram---has a homotopically unique composite, understood as a point in this contractible space.\footnote{The above discussion extends to 1-computads: Joyal observes that for any 1-skeletal simplicial set $G$, the unit component $\cwtoo{\eta_G\colon G}{N\ho{G}}$ of the nerve $\vdash$ homotopy category adjunction is inner anodyne.}

\subsection{Homotopically unique composition in an \texorpdfstring{$(\infty,2)$}{(infinity,2)}-category}

To state and prove an analogous result for pasting composition in an $(\infty,2)$-category, we must first introduce 2-dimensional pasting schemes, which may be presented by finite connected directed plane graphs with a certain property of ``anchorability'' (in terminology we borrow from Johnson and Yau). After recalling the definition in \S\ref{ssec:pasting-schemes}, in \S\ref{ssec:sub-pasting} we describe various sub pasting schemes of a given pasting scheme that are full in some sense. These are used in \S\ref{ssec:2cat} to give a more explicit description of the free 2-category $\mF\cP$ generated by a pasting scheme $\cP$ than we have found in the literature. For instance, it follows easily from Power's theorem about unique composites of pasting schemes (\cref{thm:power}) that the hom category $\mF\cP(x,y)$ between any two objects in a pasting scheme is a poset (\cref{cor:hom-poset}), which we then identify with a full sublattice of a cube whose dimension is determined by the number of atomic 2-cells that lie between the  top-most and bottom-most paths from the source object to the target one---see \cref{thm:pasting-poset} for a precise statement. This explicit characterization makes it easy to understand the effect of adding a new atomic 2-cell to the top or bottom of a pasting scheme, the procedure which forms the basis of an inductive proof of our main results.

As is the case for $(\infty,1)$-categories, the various models of $(\infty,2)$-categories build in the requirement that certain pasting schemes have homotopically unique composites. For instance, the pasting schemes that can be decomposed as a horizontal composite of vertical composites of atomic 2-cells define the objects of a category $\TTheta_2$. Rezk's $\TTheta_2$-spaces are  simplicial presheaves on this category for which these particular pasting schemes are required to have a homotopically unique composite. This suggests that the proof that a generic pasting scheme has a homotopically unique composite can be simplified by the judicious choice of a model for $(\infty,2)$-categories.

After some experimentation, we chose to work with Lurie's model of $(\infty,2)$-cat\-e\-gories as categories enriched over quasi-categories, extending Bergner's model of $(\infty,1)$-cat\-e\-gories as categories enriched over Kan complexes. An advantage of this model is that when a cell is attached in a single hom of a simplicially enriched category, its horizontal composites are automatically attached as well. So there is a sense in which this model's encoding of the 2-graph formed by the cells of a pasting scheme is closer to the free 2-category generated by the pasting scheme than it would be in other models. 

In \S\ref{ssec:simplicial-model}, we review the essential features of this model of $(\infty,2)$-cat\-e\-gories and use it to give a precise statement of our main theorem about the homotopical uniqueness of pasting composition. A pasting scheme $\cP$ can be understood as a special case of a \textbf{2-computad}: its data involves a directed graph of objects and atomic 1-cells together with a collection of atomic 2-cells with specified source and target paths. Thus, we define \textbf{pasting diagram}  in an $(\infty,2)$-category modeled as a category enriched over quasi-categories to be a simplicially enriched functor out of the free simplicially enriched category $\mG\cP$ defined by gluing together the objects, atomic 1-cells, and atomic 2-cells of a pasting scheme (see \cref{defn:graph-of-pasting-scheme}). 

By contrast, the homotopy coherent diagram generated by the pasting scheme $\cP$ is  indexed by the nerve of the free 2-category $\mF\cP$ generated by $\cP$. In the quasi-categorically enriched categories model of $(\infty,2)$-categories, the nerve functor from 2-categories to $(\infty,2)$-categories is defined by applying the ordinary nerve functor $N$ from 1-categories to $(\infty,1)$-categories to each hom\footnote{There are subtleties in identifying the ``correct'' nerve functor from 2-categories to $(\infty,2)$-categories in a particular model. Although relatively unusual, in this model of $(\infty,2)$-categories the strict point-set level nerve coincides with the homotopical one. Further discussion of this point appears in \cite{MOR:nerves}.}; following standard notational conventions, we write $N_*\mF\cP$ for the resulting quasi-categorically enriched category.  Thus, to prove the homotopical uniqueness of pasting composition, our task is to demonstrate that the canonical inclusion $\mG\cP \to N_*\mF\cP$ is a trivial cofibration in the model structure for $(\infty,2)$-categories reviewed in \cref{def sCat model structure}.

This is the claim made explicit by \cref{thm:uniquepasting}, whose proof is by an inductive argument developed in \S\ref{ssec:base-cases} and \S\ref{ssec:general-inductive-step}. In \S\ref{ssec:model-independence}, we explain how this result implies a similar equivalence in any model of $(\infty,2)$-categories as defined by Barwick and Schommer-Pries \cite{BSP}. In particular, we define the space of composites of a given pasting scheme in an $(\infty,2)$-category and prove that it is a contractible $\infty$-groupoid.

The induction that proves \cref{thm:uniquepasting} considers the effect of attaching a new atomic 2-cell along the bottom of a pasting scheme. We divide our analysis of the inductive step into two cases, depending on whether the codomain of this attached 2-cell is an atomic 1-cell or a composite of atomic 1-cells. In both cases, the simplicially enriched categories $\mG\cP$ and $N_*\mF\cP$ have the same objects so it remains only to analyze their homs. The procedure of attaching a new atomic 2-cell has a much more dramatic effect on $N_*\mF\cP$ than on $\mG\cP$; essentially the difference is between the pushout of nerves of categories---these being the hom-posets $\mF\cP(x,y)$ analyzed in \S\ref{ssec:2cat}---and the nerve of the category defined by the pushout. In \S\ref{ssec:dwyer}, we observe that one of the functors in the span that defines the pushout under consideration is a \textbf{Dwyer map}, in terminology introduced by Thomason \cite{ThomasonModelCat}. 
In a companion paper \cite{HORR-Dwyer}, we prove that the comparison map between the pushout of the nerves of a span in which one functor is a Dwyer map and the nerve of the pushouts is a weak categorical equivalence of simplicial sets, extending a related observation made by Thomason. We state this result, which may be of independent interest to technicians working with $(\infty,1)$- or $(\infty,2)$-categories, in \S\ref{ssec:dwyer-pushout} and sketch its proof.

\subsection{Related work}

After posting this preprint to the arXiv, we learned about closely related work by Tobias Columbus in an unpublished 2017 PhD thesis \cite{Columbus:2-cat}. His Lemma 2.2.10 achieves a similar description of the free 2-category generated by a pasting scheme as appears in our \cref{thm:pasting-poset} and notes in his Remark 3.1.7 the result that we prove in \cref{prop:pasting-computad}, that the homwise nerves of such 2-categories are cofibrant simplicial categories. He uses these results to prove his Theorem C,  demonstrating that any inclusion of ``complete pasting diagrams''---which, roughly, refer to pasting diagrams with certain specified composites of subdiagrams---define homwise inner anodyne inclusions. Our \cref{thm:uniquepasting} corresponds to the special case of the inclusion of a minimal complete pasting diagram $\mG\cP$, with no specified composites, into the maximal one $N_*\mF\cP$, with all composites specified. While these results are quite similar our proof techniques differ substantially. His proof of \cref{thm:uniquepasting}  makes use of the combinatorial characterization of \cref{thm:pasting-poset}, while our proof does not use this result, instead relying on the results about Dwyer maps that we describe in \S\ref{sec:dwyer}. 

Finally, Columbus concludes in Theorem D that the space of composites of a pasting scheme in a quasi-categorically enriched category is a contractible Kan complex. A posteriori, of course, his space is equivalent to the contractible $\infty$-groupoid defined in our \cref{cor:contractible-composites}, though it is tricky to give a direct comparison; see \cref{rmk:contractible-composites} for more discussion.

\section{Pasting schemes and their free 2-categories}\label{sec:pasting}

In this section, we review classical results concerning pasting composition in a 2-category or bicategory. 

We introduce our precise notion of \emph{pasting scheme} in \S\ref{ssec:pasting-schemes} and review its key properties. Our definition differs superficially from but fundamentally coincides with those appearing in \cite{PowerPasting} and \cite{JY,JYbook}; see \cref{rmk:jy-comparison} and \cref{rmk:power-comparison}. There are other related definitions and approaches to pasting diagrams for higher categories, including Johnson's pasting diagrams \cite{Johnson89}, Street's parity complexes \cite{Street:Parity}, Steiner's directed complexes \cite{Steiner93}, Steiner's augmented directed chain complexes \cite{SteinerEmbedding}, Hadzihasanovic's
constructible directed complexes \cite{Hadzihasanovic}, Forest's torsion-free complexes \cite{Forest:Unifying}, and so on. Some of these have the advantage of being entirely combinatorial, without making use of a planar embedding, but have the disadvantage of being less visually intuitive.

In \S\ref{ssec:sub-pasting}, we identify various sub pasting schemes of a given pasting scheme, for instance those defined by specifying either the source and target vertices or the source and target paths. Finally, in \S\ref{ssec:2cat}, we regard a pasting scheme as a \emph{$2$-computad} and consider the free 2-category it generates. We observe that Power's pasting theorem tells us that the hom-categories in this 2-category are all posets and we identify those posets with full subposets of cubes. 

\subsection{Pasting schemes}\label{ssec:pasting-schemes}

Consider a plane graph, by which we mean a finite connected directed graph $G$, embedded in the plane $\CC$. The connected components of $\CC \backslash G$ are the \textbf{open faces}. Their closures are called \textbf{faces}. The \textbf{boundary} of a face is the complement of the open face in the closed face, which coincides exactly with the components of the graph that touch that face. The unbounded component is called the \textbf{exterior face} while the remaining faces are \textbf{interior faces}.

By \cite[\textsection10.2]{BondyMurty}:
\begin{quote}
    In a connected plane graph, the boundary of the face can be regarded as a closed walk, in which each cut edge of the graph that lies in the boundary is traversed twice.
\end{quote}
Here a \textbf{closed walk} is an undirected cycle, which may pass through a given edge or vertex in the graph multiple times, or not at all. A \textbf{cut edge} is an edge whose removal increases the number of connected components of the graph. The cut edges are incident to a single face, either the exterior face or an interior face. The non-cut edges are incident to exactly two faces. Since our plane graph is directed, we can use the orientation of the plane to refer to one of these faces as the ``right face'' and the other face as the ``left face.''

\begin{defn}\label{defn:exterior-anchorable} The exterior face in a plane graph $G$ is \textbf{anchorable} if the set of edges that are incident to the exterior face on their left and the set of edges that are incident to the exterior face  on their right both form non-empty directed paths from a source vertex $s_G$ to a target vertex $t_G$. 
We refer to the first path as the \textbf{source path} $\dom_G$ of the graph and the second path as the \textbf{target path} $\cod_G$.
\end{defn}

\begin{defn}\label{defn:interior-anchorable} An interior face $F$ in a plane graph is \textbf{anchorable} if the edges that are incident to the face $F$ on their right and the set of edges that are incident to the face $F$ on their left both form non-empty directed paths from a vertex $s_F$ to a vertex $t_F$ which are moreover ``internally disjoint.''\footnote{Two paths with a common source and target are \textbf{internally disjoint} in the sense defined in \cite[\textsection5.1]{BondyMurty}, if they have no common edges and no common interior vertices. In particular this requirement prohibits the boundary of an interior face from having any cut edges.} We refer to the first path as the \textbf{source path} $\dom_F$ of the face $F$ and the second path as the \textbf{target path} $\cod_F$.
\end{defn}

In particular, if an interior face $F$ is anchorable, then the boundary of that face defines an undirected simple cycle embedded as a simple closed curve in the plane, with no self-intersections. This cycle can be traversed by first traversing $\dom_F$ in the forward direction and then traversing $\cod_F$ in the backwards direction.

The source path and target path of an anchorable exterior face also have no self-intersections; individually, these paths are injectively embedded in the plane. But these paths may have cut edges of the graph in common, and even if the graph contains no cut edges, these paths may intersect at an interior vertex. 

A plane graph is \textbf{anchorable} if all of its faces are anchorable, as specified by \cref{defn:exterior-anchorable} and \cref{defn:interior-anchorable}.  For later use we note the following partitionings of the edges of an anchorable plane graph:

\begin{lem}\label{lem:partitioning}
The edges of an anchorable plane graph can be partitioned in the following two ways:
\begin{enumerate}[label=(\roman*)]
    \item\label{itm:right-partition} In one partition, the components are the source paths of the interior faces plus the target path of the exterior face.
    \item\label{itm:left-partition} In the other partition, the components are the target paths of the interior faces plus the source path of the exterior face.
\end{enumerate}
\end{lem}

Hence: 
\begin{itemize}
    \item Each edge is a source of at most one interior face, and if it is not in the source path of any interior face it is in the exterior target path.
    \item Each edge is a target of at most one interior face, and if it is not in the target path of any interior face, it is in the exterior source path.
    \item The exterior source edges are either sources of exactly one interior face or they are also exterior targets.
    \item The exterior target edges are either targets of exactly one interior face or they are also exterior sources.
\end{itemize}

Note by  \cite[\textsection10.2]{BondyMurty} that the edges that are both exterior sources and exterior targets --- those edges that are incident only to the exterior face --- are exactly the cut edges, since the anchorability condition prohibits cut edges from being incident to interior faces. 

\begin{proof}
Each edge is incident to its right-hand face and its left-hand face, and no other faces. The partitions \ref{itm:right-partition} and \ref{itm:left-partition} partition the edges of the anchorable plane graph according to their right-hand faces and left-hand faces, respectively. If the right-hand face of an edge is an interior face, then by \cref{defn:interior-anchorable} the edge is in the source path of that face. If the right-hand edge is an exterior face, then by \cref{defn:exterior-anchorable} the edge is in the exterior target path. This establishes the right-hand face partition \ref{itm:right-partition}; the left-hand face partition \ref{itm:left-partition} is similar.
\end{proof}

In a directed graph, a vertex is a \textbf{local source} if it is only incident to outgoing edges and a \textbf{local sink} if it is only incident to incoming edges.

\begin{defn}\label{defn:pasting-scheme}
A \textbf{pasting scheme} is a finite connected anchorable plane graph in which the source vertex of its exterior face is the only local source  in the directed graph and the target vertex of its exterior face is the only local sink in the directed graph.
\end{defn}

In the next remark, we provide several non-examples of pasting schemes, while examples are provided below in \cref{ex: theta 2,ex: pasting schemes}.
We could also consider the graph with a single vertex and no edges as a pasting scheme, which strictly speaking is excluded by the requirement that $\dom_G$ and $\cod_G$ are non-empty in \cref{defn:exterior-anchorable}. We leave it to the reader to consider the statements of our results in this trivial case.

\begin{rmk}\label{rmk:jy-comparison} Our notion of ``anchorable'' face is very similar to Johnson and Yau's notion of ``anchored'' face \cite[Definition 3.2.6]{JYbook}; the only difference is that for our interior faces to be anchorable we require in addition that the source and target paths are internally disjoint.

Our notion of pasting scheme differs from Johnson and Yau's in two respects. Firstly, we explicitly require that the source and target vertices of the exterior face are the unique sources and sinks in the directed graph. We believe this is an implicit assumption in \cite{JYbook} but prefer to make it explicit so as to exclude pathological examples like:
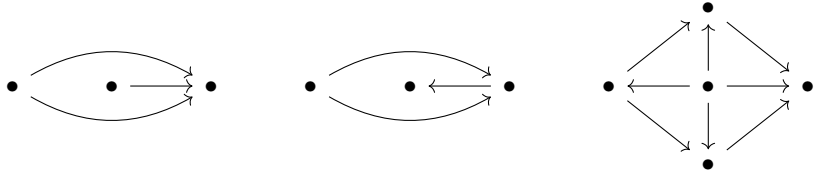
\[
\begin{tikzcd} 
& & & & & & & \bullet \arrow[dr]\\ 
\bullet \arrow[rr, bend left] \arrow[rr, bend right] & \bullet \arrow[r] & \bullet & \bullet \arrow[rr, bend left] \arrow[rr, bend right] & \bullet  & \arrow[l] \bullet & \bullet \arrow[ur] \arrow[dr] & \bullet \arrow[u] \arrow[l] \arrow[r] \arrow[d] & \bullet \\
& & & & & & & \bullet \arrow[ur]
\end{tikzcd}
\]

More significantly, we drop the data of what Johnson and Yau call a \emph{pasting scheme presentation}, which is more natural to consider in the bicategorical context than it is here, but our pasting schemes always admit a pasting scheme presentation, as we demonstrate in  \cref{cor:pasting-scheme-presentations}.
\end{rmk} 

\begin{rmk}\label{rmk:power-comparison} Power defines a pasting scheme to be a finite directed plane graph that contains no directed cycles that is equipped with distinct vertices $s$ and $t$ along its exterior face so that for every other vertex $v$ there exist directed paths from $s$ to $v$ and $v$ to $t$. These properties follow from our definition --- see \cref{lem:acyclicity} and \cref{lem:global-min-max} --- and indeed Power proves in \cite[Proposition 2.6]{PowerPasting} that his definition is equivalent to the one given here. 
\end{rmk}

We adopt the following notation and terminology for the data of a pasting scheme with the aim of connecting it to the free 2-category it generates.

\begin{ntn}
The data of a pasting scheme $\cP$ includes:
\begin{itemize}
    \item A finite collection of
    \begin{itemize}
        \item \textbf{objects} --- the vertices of the plane graph
        \item \textbf{atomic 1-cells} --- the edges of the plane graph
        \item \textbf{atomic 2-cells} --- the interior faces of the plane graph
    \end{itemize}
    \item Two distinguished objects --- the source and sink vertices $s_\cP$ and $t_\cP$ that are the source and target of the exterior face.
    \item A \textbf{source path} of composable 1-cells --- the edges of $\dom_\cP$ --- and a \textbf{target path} of composable 1-cells --- the edges of $\cod_\cP$.
    \item Each atomic 2-cell $\alpha$ also has its source and target objects $s_\alpha$ and $t_\alpha$ and source and target paths $\dom_\alpha$ and $\cod_\alpha$ of 1-cells. 
\end{itemize}
We draw a double arrow in the interior region of each interior face pointing from the source path to the target path:
\[
\begin{tikzcd}
\bullet \arrow[loop left, phantom, near end, "\scriptstyle s_\cP=s_\alpha" ]  \arrow[r, bend left] \arrow[r, bend right] \arrow[r, phantom, "\Downarrow\scriptstyle\alpha"] & \bullet \arrow[loop above, phantom, near end, "\scriptstyle t_\alpha"] \arrow[r ] & \bullet \arrow[loop below, phantom, near end,  "\scriptstyle s_\beta=s_\gamma"]  \arrow[rr, bend right=45, "\displaystyle\Downarrow\scriptstyle\gamma" {yshift=1pt}]  \arrow[r, phantom, "\Downarrow\scriptstyle\beta"] \arrow[r, bend left] \arrow[r, bend right] &  \bullet \arrow[loop above, phantom, near end, "\scriptstyle t_\beta"] \arrow[r] & \bullet \arrow[loop right, phantom, near start, "\scriptstyle t_\gamma= t_\cP"]
 \end{tikzcd}
 \]
 \cref{lem:partitioning} implies that these double arrows are ``compatibly oriented'': if an edge lies between two interior faces one double arrow will point towards that edge while the other will point away from it.
\end{ntn}

Following the convention used in the literature \cite{PowerPasting,JY}, we write paths as concatenations of edges in diagrammatic order (``left to right'') as opposed to composition order (``right to left'').

\begin{ex}\label{ex: theta 2}
The objects of Joyal's category $\TTheta_2$ define pasting schemes. In the notation introduced by Berger \cite{Berger-Wreath}, $[n]([k_1],\ldots, [k_n])$ corresponds to the pasting scheme with 
\begin{itemize}
\item $n+1$ objects $0, 1,\ldots, n$ and 
\item $k_j+1$ atomic 1-cells $e_{j,0},\ldots, e_{j,k_j}$ connecting the vertex $j-1$ to the vertex $j$.
\item $k_1+\cdots +k_n$ atomic 2-cells, each with a single source edge $e_{j,i}$ and a single target edge $e_{j,i+1}$.
\end{itemize}
If $k_j=0$ then the edge $e_{j,0}$ is a cut edge, but no other edges are cut edges. The source and target paths of the exterior face are given by $e_{1,0}\cdots e_{n,0}$ and $e_{1,k_1}\cdots e_{n,k_n}$. Atypically, these paths intersect at every vertex.

It would be straightforward to define a formula for an explicit embedding of this directed graph into the plane, but instead we simply illustrate by drawing a plane graph associated to the object $[4]([2],[0],[3],[0])$.
\[
\begin{tikzcd}[column sep=3em] \bullet  \arrow[r, "\displaystyle\Downarrow", "\displaystyle\Downarrow"'] \arrow[r, bend left=60] \arrow[r, bend right=60] & \bullet \arrow[r]& \bullet \arrow[r, bend left=25] \arrow[r, bend right=25]  \arrow[r, bend right=85, "\displaystyle\Downarrow"] \arrow[r, bend left=85, "\displaystyle\Downarrow"'] \arrow[r, phantom, "\displaystyle\Downarrow"] & \bullet \arrow[r] & \bullet
\end{tikzcd}
\]
\end{ex}

\begin{ex}\label{ex: pasting schemes}
Other pasting schemes are depicted below:
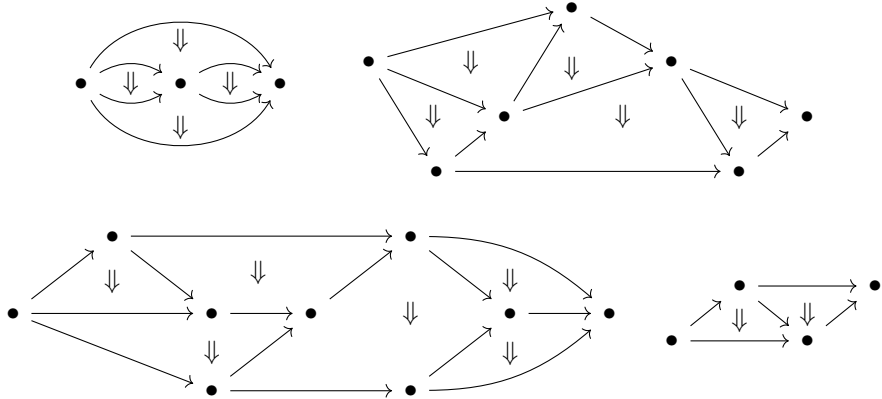
\[
\begin{tikzcd}
\bullet \arrow[r, bend left] \arrow[r, bend right] \arrow[rr, bend left=60, "\displaystyle\Downarrow"'] \arrow[rr, bend right=60, "\displaystyle\Downarrow"] \arrow[r, phantom, "\displaystyle\Downarrow"] & \bullet  \arrow[r, phantom, "\displaystyle\Downarrow"]  \arrow[r, bend left] \arrow[r, bend right] & \bullet
\end{tikzcd}
\qquad
\begin{tikzcd}[sep=small] & \arrow[ddr, phantom, "\displaystyle\Downarrow"]  & & \bullet \arrow[drr] \\
\bullet \arrow[ddr] \arrow[drr] \arrow[urrr] & & & & &  \bullet \arrow[drr] \arrow[ddr] & \arrow[dd, phantom, "\displaystyle\Downarrow"]  \\ \arrow[rr, phantom, "\displaystyle\Downarrow"] 
& & \bullet\arrow[uur] \arrow[urrr] \arrow[rrrr, phantom, "\displaystyle\Downarrow"] & \arrow[uu, phantom, "\displaystyle\Downarrow"]   & &~ & ~ & \bullet \\ & \bullet \arrow[ur] \arrow[rrrrr]  & & & & &  \bullet \arrow[ur]
\end{tikzcd}
\]
\[
\begin{tikzcd}[row sep=1.75em] & \bullet \arrow[dr, description] \arrow[rrr] \arrow[d, phantom, "\displaystyle\Downarrow"] & \arrow[dr, phantom, "\displaystyle\Downarrow"] & & \bullet \arrow[dr] \arrow[drr, bend left=20] & \arrow[d, phantom, "\displaystyle\Downarrow" pos=.6]  \\ \bullet \arrow[ur] \arrow[rr] \arrow[drr] &~ & \bullet \arrow[d, phantom, "\displaystyle\Downarrow"]  \arrow[r] & \bullet \arrow[ur] \arrow[rr, phantom, "\displaystyle\Downarrow"] & & \bullet \arrow[r] & \bullet \\ &&  \bullet \arrow[ur] \arrow[rr] & & \bullet \arrow[ur] \arrow[urr, bend right=20] & \arrow[u, phantom, "\displaystyle\Downarrow" pos=.6] 
\end{tikzcd}
\quad
\begin{tikzcd}[sep=small] & \bullet \arrow[dr] \arrow[rr] \arrow[d, phantom, "\displaystyle\Downarrow"]  & \arrow[d, phantom, "\displaystyle\Downarrow"] & \bullet \\ \bullet \arrow[ur] \arrow[rr] & ~ & \bullet \arrow[ur]
\end{tikzcd}
\]
\end{ex}

We make use of one of the central features of pasting schemes
\cite[Proposition 2.6]{PowerPasting}.

\begin{lem}[acyclicity of pasting schemes]\label{lem:acyclicity} The directed graph underlying a pasting scheme contains no directed cycles.
\end{lem}
\begin{proof}
If a directed cycle in a pasting scheme intersects itself, then that cycle contains smaller directed cycles that define simple closed curves. So we only consider directed cycles that are simple closed curves.

Choose a directed cycle that surrounds the fewest number of faces. Since an anchored face cannot be bounded by a directed cycle, some portion of the graph must lie in the region inside of the cycle.

If the region inside of the cycle contains only edges, then each edge must have source and target vertex along the cycle. Any one of these edges partitions the cycle region into two smaller regions and the boundary of exactly one of those two regions will be a directed cycle. This contradicts our hypothesis that the cycle was chosen to surround the fewest number of faces.

Thus, there must be a vertex in the region surrounded by the cycle. That vertex and every other vertex in the interior region of the cycle must have at least one incoming and at least one outgoing edge, because a pasting scheme necessarily contains  exactly one source vertex and exactly one sink vertex and these are required to lie on the exterior of the graph. Pick one interior vertex and consider any directed path that passes through it. That directed path can be extended so that it either intersects itself inside the cycle, which again contradicts minimality of the cycle, or hits the boundary of the cycle in both directions. But now this directed path partitions the cycle region into two smaller regions and the boundary of exactly one of those two regions will be a directed cycle. This contradicts our hypothesis that the cycle was chosen to surround the fewest number of faces.
\end{proof}

\begin{defn} We define an ordering of the vertices in an anchorable plane graph by declaring that $x \peq y$ if and only if there is a directed (possibly empty) path of edges from $x$ to $y$. 
\end{defn}

This ordering is reflexive and transitive by definition and can easily be seen to be antisymmetric. Any directed paths witnessing the relations $x \peq y$ and $y \peq x$ form a cycle, which contradicts  the acyclicity proven in \cref{lem:acyclicity},  unless $x=y$. Thus $\peq$ is a partial order. We write $x \prec y$ to mean that $x \peq y$ and $x \neq y$.

We now prove that the source and target vertices for a pasting scheme define initial and terminal vertices for the partial order $\peq$.

\begin{lem}\label{lem:global-min-max} Let $s_\cP$ and $t_\cP$ denote the source and target vertices for the exterior face of a pasting scheme $\cP$. Then for any vertex $v$ we have $s_\cP \peq v \peq t_\cP$.
\end{lem}
\begin{proof}
If $v \neq t_\cP$ then by \cref{defn:pasting-scheme}, the vertex $v$ has an outgoing edge. The target of that outgoing edge defines another vertex $w$. If $w= t_\cP$ we have constructed our desired path. Otherwise $w \neq t_\cP$ and has an outgoing edge, and we can repeat this process. By finiteness, either we eventually hit $t_\cP$ or we form a directed cycle, which is prohibited by \cref{lem:acyclicity}. The path from $s_\cP$ to $v$ is obtained similarly.
\end{proof}

\subsection{Sub pasting schemes}\label{ssec:sub-pasting}

In this section we prove that any pasting scheme contains various sub pasting schemes that are full on atomic 2-cells and in certain cases also on atomic 1-cells. See \cref{prop:parallel-arrow-pasting-scheme},  \cref{cor:sub-pasting-scheme}, and \cref{cor:delete-inner-face} for precise statements. We start with a series of definitions and lemmas that will support these results. 

In a plane graph, the edges incident to a fixed vertex $v$ inherit a clockwise cyclic ordering from the orientation of the plane. The following lemma tells us that in a pasting scheme, the incoming edges and outgoing edges do not mix in this cyclic order.
\[
\begin{tikzcd}[row sep=small] & \bullet \arrow[dr] \ar[dl, dotted, bend right=20, no head] & & \bullet \ar[dr, dotted, bend left=20, no head] \\ \bullet \arrow[rr]  \ar[dr, dotted, bend right=20, no head] & & v\arrow[dr] \arrow[rr] \arrow[ur] & & \bullet \ar[dl, dotted, bend left=20, no head]  \\  & \bullet \arrow[ur] & & \bullet
\end{tikzcd}
\]

\begin{lem}\label{lem prohibited local configuration}
The following configuration of distinct edges, possibly with other edges in between, does not appear in the cyclically-ordered set of edges incident to a vertex $v$ in a pasting scheme. 
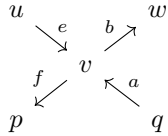
\[
\begin{tikzcd}[sep=small]
u \arrow[dr, "e"] & & w \\ 
& v \arrow[dl, "f"'] \arrow[ur, "b"]  \\ 
p  & & q \arrow[ul, "a"'] 
\end{tikzcd}
\]
\end{lem}
\begin{proof}
As the pasting scheme has no directed cycles, $w\npeq q$, $w\npeq u$, $p\npeq q$, and $p\npeq u$.
  Consider the set of pairs of paths $(\beta,\delta)$, where $\beta$ begins at $w$, $\delta$ begins at $p$, and $\beta$ and $\delta$ end at a common vertex;
 this set is non-empty, as $w\peq t_\cP$ and $p\peq t_\cP$.
Choose $(\beta_0,\delta_0)$ so that the sum of the lengths of the paths is minimal; notice that by minimality $\beta_0$ and $\delta_0$ can only intersect at the end vertices (if $w=p$ then these paths have length zero).
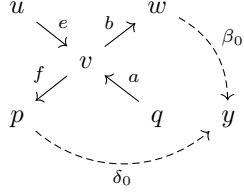
\[
\begin{tikzcd}[sep=small]
u \arrow[dr, "e"] & & w \arrow[ddr, bend left=30, dashed, "\beta_0"] \\ 
& v \arrow[dl, "f"'] \arrow[ur, "b"]  \\ 
p \arrow[rrr, bend right=40, dashed, "\delta_0"']  & & q \arrow[ul, "a"']  & y 
\end{tikzcd}
\]
The vertices $v$, $q$, and $u$ cannot appear on $\beta_0$ or $\delta_0$ without creating a directed cycle. 
After removing the Jordan curve $b\beta_0 \delta_0^{-1} f^{-1}$ from the plane, we are left with two components, one of which contains $q$ and the other of which contains $u$.
Notice that $s_\cP$ is not on the Jordan curve $b\beta_0 \delta_0^{-1} f^{-1}$, and without loss of generality, assume that $s_\cP$ is in the same component of the complement as $u$.
Then any path from $s_\cP$ to $q$ in the plane graph must intersect $b\beta_0$ or $f\delta_0$; we know that $v$ will not be on such a path since $v\npeq q$.
It follows that a path from $s_\cP$  to $q$ must intersect $\beta_0$ or $\delta_0$ at some vertex $r$, implying $p\peq r \peq q$ or $w\peq r \peq q$, contradicting $p\npeq q$ or $w\npeq q$.
\end{proof}

\cref{lem prohibited local configuration} has a useful consequence: for any vertex $v$ in a pasting scheme, the set of incoming edges inherits a natural total order from the orientation imposed by the plane, and similarly the set of outgoing edges inherits a natural total order. We fix the following conventions for the directions of these orderings:

\begin{cvn}\label{cvn:inp-out-order} Let $\cP$ be a pasting scheme, and $v$ a vertex in $\cP$. Let $\inp(v)$ be the set of edges with target $v$, and $\out(v)$ be the set of edges with source $v$. We regard these as totally ordered sets as follows:
\begin{itemize}
\item $\out(v) = \{ e_1 \leq e_2 \leq \cdots \leq e_n \}$ where $e_i$ appears immediately before $e_{i+1}$ as one traverses counterclockwise. If $v= s_\cP$, then we additionally must specify that $e_1 \in \cod_\cP$ or $e_n \in \dom_\cP$.
\item $\inp(v) = \{ e_1' \leq e_2' \leq \cdots \leq e_m' \}$ where $e_i'$ appears immediately before $e_{i+1}'$ as one traverses clockwise. If $v= t_\cP$, then we additionally must specify that $e_1' \in \cod_\cP$ or $e_m' \in \dom_\cP$.
\end{itemize}
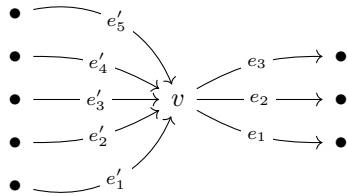
\[ \begin{tikzcd}[row sep=tiny]
\bullet \ar[ddrr, bend left=40, "e_5'" description] \\
\bullet \arrow[drr, bend left=15, "e_4'" description] &  & &  &\bullet \\
\bullet \arrow[rr, "e_3'" description] & & v \ar[rr, "e_2" description] \ar[urr, bend left=15, "e_3" description] \ar[drr, bend right=15, "e_1" description] &  & \bullet  \\
\bullet  \arrow[rru, bend right=15, "e_2'" description] &  & &   & \bullet \\
\bullet \ar[uurr, bend right=40, "e_1'" description]
\end{tikzcd} \]
We additionally write $\inp(v)_+$ and $\out(v)_+$ for the partially-ordered sets where we have added an incomparable element $\ast$.
\end{cvn}

The idea of both orderings is that we write $e \geq e'$ if $e$ and $e'$ share a common source or a common target and the edge $e$ lies ``above'' the edge $e'$. But in fact we use \cref{lem prohibited local configuration} and the convention just introduced to give a precise definition of what it means for one edge or path to lie above another. For this, we require:

\begin{defn}[path predecessor and successor]
Suppose that $p$ is a path in a pasting scheme.
If $v$ is an arbitrary vertex, define elements $\pred(p,v) \in \inp(v)_+$ and $\suc(p,v)\in \out(v)_+$
by
\begin{itemize}
\item $\pred(p,v) = e$ just when $e \in \inp(v)$ lies on the path $p$,
\item $\suc(p,v) = e$ just when $e \in \out(v)$ lies on the path $p$,
\item and otherwise these are defined to be the null points $\ast$.
\end{itemize}
\end{defn}

We indicate a few instances of the preceding definition in the following example, where we consider the dashed path $p$. Note that $\pred(p,s) = \ast$.
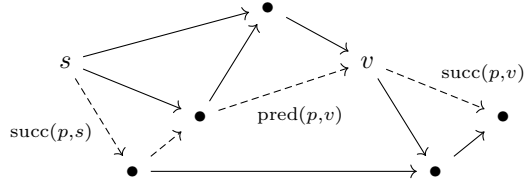
\[
\begin{tikzcd}[sep=small] & & & \bullet \arrow[drr] \\
{ s} \arrow[ddr, dashed,  "{\suc(p,s)}"'] \arrow[drr] \arrow[urrr] & & & & &  {v} \arrow[drr, dashed,  "{\suc(p,v)}"] \arrow[ddr] \\
& & {\bullet} \arrow[uur] \arrow[urrr, dashed,  "{\pred(p,v)}"' near start] & & & & & \bullet \\
 & {\bullet} \arrow[ur,  dashed] \arrow[rrrrr]  & & & & &  \bullet \arrow[ur]
\end{tikzcd}
\]

The partial order introduced in \cref{cvn:inp-out-order} allows us to identify when a path lies above another parallel path.
\begin{defn}\label{defn:lies above}
Suppose that $p$ and $q$ are paths from $x$ to $y$ in a pasting scheme.
If, for each vertex $v$ which appears on both $p$ and $q$ we have $\pred(p,v) \geq \pred(q,v)$ and $\suc(p,v) \geq \suc(q,v)$, then we say that $p$ \textbf{lies above} $q$.
\end{defn}

\begin{lem}\label{lem short lies above}
If $p$ and $q$ are paths from $x$ to $y$ in a pasting scheme and $p$ and $q$ do not share any edges or vertices other than $x$ and $y$, then $p$ lies above $q$ or $q$ lies above $p$.
\end{lem}
\begin{proof} Without loss of generality, suppose $\suc(p,x) > \suc(q,x)$. We will show that $\pred(p,y) > \pred(q,y)$. Since $p$ and $q$ have no other common vertices, this implies that $p$ lies above $q$. 

Choose any paths $r$ from $s_{\cP}$ to $x$ and $r'$ from $y$ to $t_{\cP}$ so that $r \cdot p \cdot r'$ and $\cod_{\cP}$ bound a closed subspace that contains the first outgoing edge of the path $q$. If $\pred(q,y) > \pred(p,y)$ then the path $q$ must leave this region by intersecting either the path $r$ or the path $r'$. If $q$ intersects $r$ then there is a directed cycle that starts at $x$, continues along $q$ to the vertex of intersection, and returns to $x$ along the path $r$. If $q$ intersects $r'$ then there is a directed cycle that starts at $y$, continues along $r'$ to the point of intersection, and returns to $y$ along the path $q$. This is prohibited by \cref{lem:acyclicity}.
\end{proof}

It follows that we can identify when a path $p$ lies above a path $q$ by considering only the incoming edges or only the outgoing edges of their common vertices.

\begin{lem}\label{lem only test half}
Suppose that $p$ and $q$ are paths from $x$ to $y$ in a pasting scheme.
If, for each vertex $v$ which appears on both $p$ and $q$ we have $\pred(p,v) \geq \pred(q,v)$, then $p$ lies above $q$.
Likewise, if for each vertex $v$ which appears on both $p$ and $q$ we have $\suc(p,v) \geq \suc(q,v)$, then $p$ lies above $q$.
\end{lem}
\begin{proof} We prove the first statement. If $v \neq y$ is on both $p$ and $q$ we must show that $\suc(p,v) \geq \suc(q,v)$. 
If $\suc(p,v) = \suc(q,v)$, then we are done, so suppose this is not the case and let $w$ be the next point of intersection of the paths $p$ and $q$, after the vertex $v$. Then $p$ and $q$ restrict to define paths from $v$ to $w$ that do not share any edges or vertices other than $v$ and $w$. 
Since $\pred(p,w) \geq \pred(q,w)$, by \cref{lem short lies above} we know that the restriction of $p$ lies above the restriction of $q$, hence $\suc(p,v) \geq \suc(q,v)$.
\end{proof}

\begin{lem}\label{lem:paths_partition}
Let $p$ and $q$ be paths from $x$ to $y$ in a pasting scheme.
If $p$ lies above $q$, then $p$ and $q$ factor uniquely as concatenations of paths
\begin{align*}
p&= r_0 p_1 r_2 p_2 \cdots r_{n-1} p_n r_n \\
q&= r_0 q_1 r_2 q_2 \cdots r_{n-1} q_n r_n
\end{align*}
where
\begin{itemize}
\item each $p_i$ and each $q_i$ has positive length,
\item $p_i$ and $q_i$ intersect at their endpoints and do not intersect at any interior vertices,
\item the $r_i$ are paths, possibly of length zero (containing only a vertex), and 
\item $p_i$ lies above $q_i$.
\end{itemize}
\end{lem}
\begin{proof}
Induct on the number $k$ of vertices held in common between $p$ and $q$.
The base case is when $k=2$, that is where $p$ and $q$ only intersect at the beginnings and ends, in which case we can take $p_1 =p$, $q_1 = q$, and $r_0$, $r_1$ to be the extremal vertices.
If $k > 2$, choose a vertex $v$ away from the ends that lies on both paths, and write $p = p_-p_+$ and $q=q_-q_+$ as concatenations of positive length paths, breaking at $v$.
Then $p_-$ lies above $q_-$ and $p_+$ lies above $q_+$.
We have strictly fewer than $k$ vertices in common between $p_-$ and $q_-$, and likewise for $p_+$ and $q_+$, so the induction hypothesis gives the relevant splittings which we then combine.
\end{proof}

For later use, we observe that the relation of a path lying above another path is transitive.

\begin{lem} Let $p$, $q$, and $r$ be paths of edges from $x$ to $y$ in a pasting scheme. If $p$ lies above $q$ and $q$ lies above $r$ then $p$ lies above $r$.
\end{lem}
\begin{proof}
To prove that $\suc(p,v) \geq \suc(r,v)$ and $\pred(p,v) \geq \pred(r,v)$ for every vertex $v$ that belongs to both $p$ and $r$ we argue that all such vertices also lie on $q$. The claimed result then follows from the transitivity of the relation $\geq$ on $\out(v)$ and $\inp(v)$.

It suffices to consider the first vertex $v$ that is not equal to $x$ or $y$ that lies along the intersection of $p$ and $r$ in the order imposed by these directed paths. The initial segments of $p$ and $r$ then define directed paths from $x$ to $v$ that only intersect at their endpoints. Since the vertex $y$ lies outside this bounded region and $q$ is a path from $x$ to $y$ whose initial edge lies below the path $p$ and above the path $r$, the path $q$ must exit this region at some vertex $u$ that lies either on $p$ or $r$. We claim that $u=v$.

If $u$ lies on $p$ strictly between $x$ and $v$ and $u$ is the exit vertex for $q$, then $\suc(q,u) > \suc(p,u)$, contradicting the fact that $p$ lies above $q$. A similar argument shows that $u$ cannot lie along $r$ strictly between $x$ and $v$. Hence $u=v$ is a point on the path $q$.
\end{proof}

\begin{defn}\label{def:p-over-q}
Suppose that $p$ lies above $q$ in a pasting scheme $\cP$.
Using the notation from \cref{lem:paths_partition}, let $C_i \subset \CC\backslash (p_i \cup q_i)$ be the bounded region.
Write $p/q$ for the subgraph of $\cP$ containing all vertices and edges in the paths $p$ and $q$, as well as all edges and vertices appearing in $\bigcup_{i=1}^n \overline{C_i}$.
\end{defn}

\begin{prop}\label{prop:parallel-arrow-pasting-scheme}
If $p$ and $q$ are non-trivial paths from $x$ to $y$ in a pasting scheme  $\cP$ and $p$ lies above $q$, then $p/q$ is a pasting scheme in which:
\begin{itemize}
\item A vertex $v$ of $\cP$ is in $p/q$ if and only if there is some path $m$ of $\cP$ from $x$ to $y$ which lies below $p$, lies above $q$, and contains $v$. 
\item An edge $e$ of $\cP$ is in $p/q$ if and only if there is some path $m$ of $\cP$ from $x$ to $y$ which lies below $p$, lies above $q$, and contains $e$.  In particular, a path $m$ from $x$ to $y$ in $\cP$ lies in $p/q$ if and only if it lies below $p$ and lies above $q$.
\item The atomic $2$-cells of $p/q$ are those atomic $2$-cells of $\cP$ which lie in one of the discs $C_i$. Hence an atomic $2$-cell $\alpha$ lies in $p/q$ if and only the paths $\dom_\alpha$ and $\cod_\alpha$ lie in $p/q$.
\end{itemize}
Moreover, $s_{p/q} = x$, $t_{p/q} = y$, $\dom_{p/q} = p$ and $\cod_{p/q} = q$.
\end{prop}
\begin{proof}
We first show that $p/q$ is a pasting scheme.
Interior faces of $p/q$ are also interior faces of $\cP$, hence anchorable.
To see the exterior face of $p/q$ is anchorable, first note that if an edge $e$ of $p/q$ is incident to the exterior face, then $e$ must be in $p$ or be in $q$. 
The edges of $p$ are exactly those edges which are incident to the exterior face on their left, and the edges of $q$ are exactly those which are incident to the exterior face on their right.
It follows that the exterior face is anchorable, and $\dom_{p/q} = p$, $\cod_{p/q} = q$. The graph $p/q$ is finite and connected (by the characterization of vertices and edges), so the conditions of \cref{defn:pasting-scheme} are satisfied.

By \cref{lem prohibited local configuration}, no edge in $\inp(x)$ and no edge in $\out(y)$ lies in $p/q$. Hence, $x$ is a source in $p/q$ and $y$ is a sink in $p/q$.
The vertex $x$ is the unique source of the graph $p/q$:
certainly every other vertex in $p$ and every other vertex in $q$ has an input edge, while the vertices in the open region $C_i$ have exactly the same set of input edges as they had in $\cP$. A similar argument shows that $y$ is the unique sink of $p/q$.
\end{proof}

Another useful notion is the sub pasting scheme between two vertices, which arises as a special case of the preceding proposition. To prove this, we use the following:

\begin{lem}\label{lem:maximal-path-x-to-y}
If $x\prec y$ are vertices of a pasting scheme $\cP$, then there is a unique path $p$ from $x$ to $y$ which lies above all other paths from $x$ to $y$. Dually, there is a unique path $q$ which lies below all other paths from $x$ to $y$.
\end{lem}
\begin{proof}
We prove only the first statement.
Color an edge $e$ of $\cP$ blue if it lies on some path from $x$ to $y$. If $v\neq y$ and there is a blue edge in $\inp(v)$, then there is also a blue edge in $\out(v)$; one can see this because the blue edge in $\inp(v)$ is witnessed by some path from $x$ to $y$, which must continue out of $v$.
    
By assumption, $x \prec y$, hence there is a blue edge in $\out(x)$.
We form the path $p = e_1 \cdots e_n$
where $e_1$ is the greatest blue edge in $\out(x)$ and $e_{i+1}$ is the greatest blue edge in $\out(v_{i})$, where $v_i$ is the target of $e_i$. This terminates at some edge $e_n$ with target $y=v_n$ because $\cP$ is finite and has no cycles. 

To show that $p$ lies above any other path $r$ from $x$ to $y$, we appeal to \cref{lem only test half}. Because $r$ is a path from $x$ to $y$, all of its edges are blue. Thus, for any $v$ which appears on both $p$ and $r$, we have that $\suc(p,v) \geq \suc(r,v)$ by construction. Thus, \cref{lem only test half} tells us that $p$ lies above $r$.
\end{proof}

We now show that for any pair of vertices $x$ and $y$ in a pasting scheme, the subgraph spanned by those vertices $v$ so that $x \peq v \peq y$ defines a pasting scheme that is full in various senses that we now establish:

\begin{cor}\label{cor:sub-pasting-scheme}
Suppose that $\cP$ is a pasting scheme, $x \prec y$ are vertices.
Let $\cP_{x,y}$ be the induced plane subgraph containing all vertices $u$ so that $x \peq u \peq y$ and all edges these vertices span.
Then $\cP_{x,y}$ is a pasting scheme, with  source vertex $x$ and  target vertex $y$, which has
\begin{itemize}
\item as objects all vertices $u \in P$ so that $x \preccurlyeq u \preccurlyeq y$,
\item as atomic $1$-cells all directed edges between these vertices, 
\item as atomic $2$-cells all atomic $2$-cells $\alpha$ from $\cP$ so that $x \peq s_\alpha$ and $t_\alpha \peq y$, or equivalently all faces $\alpha$ whose boundary paths $\dom_\alpha$ and $\cod_\alpha$ lie in $\cP_{x,y}$.
\end{itemize}
Moreover, the pasting scheme $\cP_{x,y} \subset \cP$ 
contains all objects, atomic $1$-cells, and atomic $2$-cells of $\cP$ that are contained within the closure of the exterior boundary of $\cP_{x,y}$.
\end{cor}
\begin{proof}
Let $p$ and $q$ be the top-most and bottom-most paths from $x$ to $y$, as in \cref{lem:maximal-path-x-to-y}.
Then using the description from \cref{prop:parallel-arrow-pasting-scheme}, the graph $\cP_{x,y}$ is $p/q$.
\end{proof}

\cref{prop:parallel-arrow-pasting-scheme} and  \cref{cor:sub-pasting-scheme} give two important notions of ``sub pasting scheme'' full on those vertices that lie between a specified global source vertex $x$ and global target vertex $y$. We also make use of another notion of ``sub pasting scheme'' that removes one of the atomic 2-cells. It is not possible to remove an arbitrary atomic 2-cell and retain a pasting scheme, much like it is not possible to remove an arbitrary collection of vertices. But we are able to prune a single atomic 2-cell if its entire source path is a subpath of the exterior source path or if its entire target path is a subpath of the exterior target path. And moreover, any pasting scheme that has at least one atomic 2-cell contains some atomic 2-cells  with each of these properties. This will form the basis of our inductive arguments.

To prove this, we require a few preliminary results. The first of these allows us to analyze the adjacent edges in the source path $\dom_\cP$ for a pasting scheme $\cP$, extending \cref{lem:partitioning}.

\begin{lem}\label{lem start or end}
Let $\cP$ be a pasting scheme and let $v$ be an interior vertex of its source path $\dom_\cP$.
Suppose that the incoming edge $e \colon u \to v$ is in both $\dom_\cP$ and in $\dom_\alpha$, and the outgoing edge $e' \colon v \to w$ in $\dom_\cP$ and in $\dom_\beta$.
If $\alpha$ and $\beta$ are distinct atomic $2$-cells, then either $v = t_\alpha$ or $v = s_\beta$.
\end{lem}
\begin{proof}
If $v$ is neither the target of $\alpha$ nor the source of $\beta$ we have two further edges
$f \colon v \to p$ in $\dom_\alpha$ and $f' \colon q \to v$ in $\dom_{\beta}$. Because $e$ and $f$ lie on the domain path of $\alpha$ and $e'$ and $f'$ lie on the domain path of $\beta$, in the clockwise cyclic ordering of edges adjacent to $v$ we have $e'$ then $f'$ then $f$ then $e$---possibly with other edges in between. This is prohibited by \cref{lem prohibited local configuration}.
\end{proof}

\begin{rmk}\label{rmk:cyclic-ordering-exterior}
Let $v$ be a vertex in pasting scheme so that $v$ is on the boundary of the exterior face. Then since the exterior face is anchorable, $v$ is incident to one, two, three, or four edges on the boundary of the exterior face.
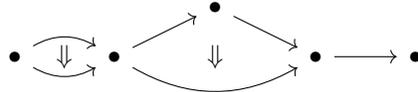
\[
\begin{tikzcd}[row sep=tiny]
& & \bullet \arrow[dr]
\\ 
\bullet \arrow[r, bend left] \arrow[r, phantom, "\Downarrow"] \arrow[r, bend right] & \bullet \arrow[rr, bend right]  \arrow[ur] & \Downarrow & \bullet \arrow[r] & \bullet
\end{tikzcd}
\]
If $s_\cP \neq v \neq t_\cP$ the source and target paths pass through $v$ at most once, so $v$ cannot be incident to five or more boundary edges.

If two of the edges are along $\dom_\cP$, then the one coming into $v$ immediately precedes the one coming out of $v$ in the clockwise cyclic ordering, and if two of the edges are along $\cod_\cP$ then the one coming out of $v$ immediately precedes the one coming into $v$ in the clockwise cyclic ordering.
\end{rmk}

Armed with these observations, we can prove the result we want, recovering Power's \cite[Proposition 2.10]{PowerPasting}:

\begin{prop}\label{prop:top-face}
If $\cP$ is a pasting scheme with at least one interior face, then there exists an atomic $2$-cell $\alpha$ so that $\dom_\alpha$ is contained in $\dom_\cP$.
\end{prop}

We refer to an atomic 2-cell $\alpha$ with $\dom_\alpha \subset \dom_\cP$ as a 2-cell that is \textbf{at the top} of the pasting scheme. The dual of \cref{prop:top-face} proves also that any pasting scheme $\cP$ with at least one interior face contains at least one atomic 2-cell $\beta$ that is \textbf{at the bottom}, meaning that $\cod_\beta \subset \cod_\cP$.

\begin{proof}
We know from \cref{lem:partitioning} that each edge in $\dom_\cP$ either belongs to $\cod_\cP$ or belongs to some interior face. Starting from $s_\cP$ we remove  from $\cP$ all of the edges along $\dom_\cP$ that belong to $\cod_\cP$ until we meet the first edge $e \colon v \to w$ that instead belongs to $\dom_\alpha$ for some atomic 2-cell $\alpha$. Then $v=s_\alpha$ because certainly $s_\alpha \peq v$ but the only vertices that precede $v$ in $\cP$ belong to $\dom_\cP \cap \cod_\cP$ and in particular do not belong to $\alpha$. Continuing along $\dom_\cP$ from $v$ we either trace the full path $\dom_\alpha$, in which case in particular we meet the vertex $t_\alpha$, or we meet a vertex $x\neq t_\alpha$ so that the next edge $f \colon x \to y$ in $\dom_\cP$ does not belong to $\dom_\alpha$. By \cref{lem:partitioning} either $f$ is a source edge for another interior face or $f$ belongs to $\cod_\cP$. 

If $f \in \cod_\cP$, then by \cref{rmk:cyclic-ordering-exterior} we have the cyclic ordering
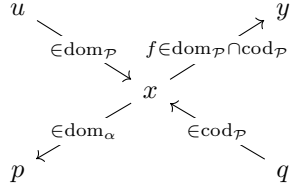
\[
\begin{tikzcd}[column sep=large]
u \arrow[dr, "\in \dom_\cP" description] & & y \\ 
& x \arrow[dl, "\in \dom_\alpha" description] \arrow[ur, "f \in \dom_\cP \cap \cod_\cP" description]  \\ 
p  & & q \arrow[ul, "\in \cod_\cP" description] 
\end{tikzcd}
\]
which is prohibited by \cref{lem prohibited local configuration}, so we must have $f \in \dom_{\alpha'}$ for another atomic 2-cell $\alpha'$.  Now by \cref{lem start or end}, we must have $x = s_{\alpha'}$.

The pasting scheme $\cP$ contains only finitely many atomic 2-cells, so this process cannot continue indefinitely, and we must eventually meet a target vertex of the face whose source path we are tracing. Thus, we must eventually find a top cell.
\end{proof}

\begin{cor}\label{cor:delete-inner-face}
If $\alpha$ is an atomic $2$-cell that is at the top of a pasting scheme $\cP$, then we can remove all edges and all interior vertices of $\dom_\alpha$ to obtain a new plane graph $\cP\backslash \alpha$.
The plane graph $\cP\backslash \alpha$ is also a pasting scheme, which contains all other vertices and atomic $1$-cells and every atomic $2$-cell of $\cP$ except $\alpha$.
\end{cor}
\begin{proof}
To say that $\alpha$ is a top 2-cell means that $\dom_\cP = r \cdot \dom_\alpha \cdot u$ for some paths $r$ and $u$. 
Define $p=r \cdot \cod_\alpha \cdot u$ and $q=\cod_\cP$. 
The path $p$ lies above $q$, since every path from $s_\cP$ to $t_\cP$ lies above $\cod_\cP$.
Then $p/q = \cP\backslash \alpha$ is a pasting scheme by \cref{prop:parallel-arrow-pasting-scheme}.
\end{proof}

Combining these two results it follows that any pasting scheme admits a ``pasting scheme presentation'' \`a la \cite[Definition 3.2.13]{JYbook}. The data of a pasting scheme presentation involves a choice, for each atomic 2-cell $\alpha$ in a pasting scheme, of a path of edges from $s_\cP$ to $s_\alpha$ and a path of edges from $t_\alpha$ to $t_\cP$, together with a specified linear ordering on the atomic 2-cells satisfying a condition we spell out in the statement:

\begin{cor}\label{cor:pasting-scheme-presentations} For any pasting scheme $\cP$ there exists a linear ordering of its atomic $2$-cells $\alpha_1 \geq \cdots \geq \alpha_n$ together with specified paths of edges $p_i$ from $s_\cP$ to $s_{\alpha_i}$ and $q_i$ from $t_{\alpha_i}$ to $t_\cP$ for each $i$ so that for each $1 \leq i < n$, the concatenated path $p_i \cdot \cod_{\alpha_i} \cdot q_i$ equals the concatenated path $p_{i+1} \cdot \dom_{\alpha_{i+1}} \cdot q_{i+1}$.
\end{cor}
\begin{proof}
If $\cP$ has no interior faces, then there is nothing to prove.
By \cref{prop:top-face}, for any pasting scheme $\cP$ which has an interior face, we may find an atomic 2-cell $\alpha_1$ with $\dom_{\alpha_1} \subset \dom_\cP$. We choose one such atomic 2-cell and make it first in our linear ordering. We define the paths $p_1$ and $q_1$ to be the two components of the path $\dom_\cP \backslash \dom_{\alpha_1}$.\footnote{If $s_\cP = s_{\alpha_1}$ then $p_1$ is empty, while if $t_{\alpha_1} = t_\cP$ then $q_1$ is empty.}

By \cref{cor:delete-inner-face}, when we remove the interior edges of the path $\dom_{\alpha_1}$ we obtain a new pasting scheme $\cP\backslash \alpha_1$ whose domain path is the concatenation $p_1 \cdot \cod_{\alpha_1}\cdot q_1$. 
If $\cP\backslash \alpha_1$ has an interior face, we iterate the procedure described above: choosing another atomic 2-cell $\alpha_2$ with $\dom_{\alpha_2}$ a subset of this path and define $p_2$ and $q_2$ to be the complementary paths; note from this definition that $p_1 \cdot \cod_{\alpha_1}\cdot q_1 = p_2 \cdot \dom_{\alpha_2} \cdot q_2$. The result follows.
\end{proof}

\begin{defn}\label{lies directly above}
We define a binary relation on the atomic 2-cells of a pasting scheme as follows: for atomic 2-cells $\alpha$ and $\beta$ we say that $\alpha$ \textbf{lies directly above} $\beta$ whenever there is a 1-cell $e \in \cod_\alpha \cap \dom_\beta$. 
\end{defn}

\begin{lem}\label{transitive lies above}
The reflexive and transitive closure of the ``lies directly above'' relation is antisymmetric, defining a partial order on the set of atomic 2-cells in a pasting scheme.
\end{lem}
\begin{proof}
Suppose $\alpha$ lies directly above $\beta$ and we are given a path $q$ from $s_{\cP}$ to $t_{\cP}$ which contains $\dom_{\beta}$.
Then, by choosing the maximal paths $r\colon s_{\cP} \to s_{\alpha}$ and $r' \colon t_\alpha \to t_{\cP}$, we obtain a path  $p = r \cdot \cod_\alpha \cdot r'$ from $s_{\cP}$ to $t_{\cP}$ that contains $\cod_\alpha$. To see that $p = r \cdot \cod_\alpha \cdot r'$ lies above $q$, let $s_{\cP} = v_0 \prec v_1 \prec \cdots \prec v_n = t_{\cP}$ be the vertices of intersection.
We know that there is an edge $e \colon v_i \to v_{i+1}$ in $\dom_\beta \cap \cod_\alpha$.
Let $p_-,q_- \colon s_{\cP} \to v_{i+1}$ and $p_+,q_+ \colon v_i \to t_\cP$ be the restrictions.
It is enough to show that $p_-$ lies above $q_-$ and $p_+$ lies above $q_+$; we only prove the latter. By \cref{lem only test half}, it is enough to show that $\suc(p, v_k) \geq \suc(q,v_k)$ for $i \leq k < n$. 
By hypothesis, $\suc(p,v_i) = e = \suc(q,v_i)$.
Let $i < k < n$, and write $e_1 = \pred(p,v_k)$, $e_2 = \suc(p,v_k)$, and $e_2' = \suc(q,v_k)$; we want to show that $e_2 \geq e_2'$.
If $e_2$ is on $\cod_\alpha$, then so is $e_1$, and whenever we have this configuration  on the codomain of a 2-cell, it is automatic that $e_2$ is maximal. Thus in this case $e_2 \geq e_2'$.
If $e_2$ is not on $\cod_\alpha$, then it is on the maximal path $r'$ from $t_{\alpha}$ to $t_{\cP}$, hence $e_2 \geq e_2'$.

The previous paragraph implies that given any chain $\alpha_1, \alpha_2, \ldots, \alpha_m$ of atomic 2-cells of $\cP$ with $\alpha_i$ directly above $\alpha_{i+1}$, we may construct 
a sequence of paths from $s_\cP$ to $t_\cP$, say
$p_1 > p_1' \geq p_2 > p_2' \cdots \geq p_m > p_m'$, with 
$\dom_{\alpha_i}$ contained in $p_i$ and $\cod_{\alpha_i}$ contained in $p_i'$ as follows. Choose $p_m'$ to complete $\cod_{\alpha_m}$ arbitrarily to a path from $s_{\cP}$ to $t_{\cP}$. 
We get $p_i$ from $p_i'$ by removing $\cod_{\alpha_i}$ and replacing it with $\dom_{\alpha_i}$.
We obtain $p_{i-1}'$ from $p_i$ using the construction of the previous paragraph.
If in our chain we have $\alpha_1 = \alpha_m$ and $m>1$, then $\cod_{\alpha_1} = \cod_{\alpha_m}$ is contained in both $p_1'$ and $p_m'$, hence is contained in the intermediate path $p_m$.
But by construction $p_m$ does not contain $\cod_{\alpha_m}$.
\end{proof}

\subsection{The hom-posets of a pasting scheme}\label{ssec:2cat}

A pasting scheme defines an instance of a \textbf{2-computad} (also called a 2-polygraph), this being a directed graph together with a collection of atomic 2-cells, each with specified source and target paths of edges in the  graph \cite{Street:Limits}. Any 2-computad generates a free 2-category, which specializes to define the free 2-category $\mF\cP$ associated to a pasting scheme $\cP$, which has 
\begin{itemize}
    \item objects defined to be the vertices of the underlying graph of $\cP$,
    \item 1-cells defined to be directed paths of edges in the underlying graph of $\cP$, and
    \item 2-cells, generated by whiskered atomic 2-cells  under vertical composition, satisfying relations which are a bit more complicated to describe.
\end{itemize}
In general, a 2-cell may be expressible in multiple ways as a vertical composite of whiskered atomics. A simple example is given by the 2-computad $[2]([1],[1])$ which generates the free horizontally composable pair of 2-cells. Here the horizontal composite can be expressed as a whiskered composite of the atomic 2-cells in two different ways.\footnote{The underlying 1-category of a free 2-category is free as a 1-category. In this case, it is the free category on the directed graph of the 2-computad. Note however that the hom-categories of a free 2-category need not be free, as this example illustrates.}

A pasting scheme $\cP$ is a very special sort of 2-computad where the underlying directed graph is a plane graph and the 2-cells are compatibly oriented cells inhabiting the interior faces. The aim in this section is to more explicitly describe the hom-categories in the  free 2-category $\mF\cP$ generated by the pasting scheme, expanding on the definition sketched above. For this, we make use of the main theorem of Power \cite{PowerPasting}.

\begin{thm}[Power]\label{thm:power} Any pasting scheme defines a unique composite $2$-cell, which can be understood as the vertical composite of the whiskered atomic $2$-cells in any pasting scheme presentation.
\end{thm}

An interpretation of Power's theorem is that there is a unique 2-cell from $\dom_\cP$ to $\cod_\cP$ in $\mF \cP$. Leveraging our understanding of sub pasting schemes of a pasting scheme $\cP$, we can deduce also:

\begin{cor}\label{cor:hom-poset} For any objects $x,y$ in $\cP$, the hom-category $\mF \cP(x,y)$ is a poset. Moreover there is a $2$-cell in $\mF \cP$ from a path $p$ from $x$ to $y$ to a path $q$ from $x$ to $y$ if and only if $p$ lies above the path $q$.
\end{cor}

Henceforth, we refer to $\mF \cP(x,y)$ as the \textbf{hom-poset} from $x$ to $y$ in $\cP$.

\begin{proof}
By definition of $\mF \cP$, the objects in $\mF \cP(x,y)$ are paths from $x$ to $y$ in $\cP$. We claim there is a morphism $p$ to $q$ in $\mF \cP(x,y)$ if and only if $p$ lies above $q$. The only arrows in $\mF \cP (x,y)$ are vertically composable sequences of whiskered atomic 2-cells. For every whiskered atomic 2-cell, the source path is above the target path. So if $p$ is not above $q$, then there can be no arrow from $p$ to $q$. 

If $p$ lies above $q$ then \cref{prop:parallel-arrow-pasting-scheme} proves that the paths $p$ and $q$ define a sub pasting scheme $p/q\subset \cP$. By \cref{thm:power}, the hom-category $\mF p/q (x,y)$ contains a unique arrow from $p$ to $q$. Since the inclusion $p/q\subset \cP$ is full on atomic 2-cells and full on paths that lie below $p$ and above $q$, it follows that $\mF \cP(x,y)$ contains a unique  arrow from $p$ to $q$ in $\cP$.
\end{proof}

\begin{cor}\label{cor:loc_ff_sub}
If $p$ lies above $q$ in a pasting scheme $\cP$, then the inclusion induces an inclusion of $2$-categories
$\mF p/q \to \mF \cP$ which is locally fully faithful:  $\mF p/q (x,y) \to \mF \cP (x,y)$ is a fully-faithful inclusion of posets for all vertices $x$ and $y$ in $p/q$.
\end{cor}
\begin{proof}
This 2-functor is evidently an inclusion, and the local fullness comes from \cref{prop:parallel-arrow-pasting-scheme}.
\end{proof}

\begin{prop}\label{prop:concat_ff}
If $\cP$ is a pasting scheme and $x,y,z$ are objects, then the composition map
\[
\begin{tikzcd}  \mF \cP(y,z) \times \mF \cP(x,y) \arrow[r, "\circ", hook] & \mF \cP(x,z)
\end{tikzcd}
\]
is a fully-faithful inclusion of posets. 
\end{prop}
\begin{proof}
Assume $x \peq y$ and $y\peq z$ otherwise there is nothing to prove.
Let $p_-$ be the path from $x$ to $y$ lying above all other paths from $x$ to $y$ and $q_-$ be the path from $x$ to $y$ lying below all other paths from $x$ to $y$.
Likewise let $p_+$ and $q_+$ the top-most and bottom-most paths from $y$ to $z$.
Let $p = p_- \cdot p_+$ and $q = q_- \cdot q_+$ be the composite paths from $x$ to $z$.
Any path in $p/q$ from $x$ to $z$ contains the vertex $y$, 
 and splitting at $y$ gives a functor from the maximal hom poset $\mF p/q (x,z)$ to $\mF \cP(y,z) \times \mF \cP(x,y) = \mF p_+/q_+(y,z) \times \mF p_-/q_-(x,y)$.
This is a bijection of objects and admits an inverse map of posets, so identifies 
$\mF \cP(y,z) \times \mF \cP(x,y)$ with $\mF p/q (x,z)$.
We then apply \cref{cor:loc_ff_sub} to deduce the result.
\end{proof}

We now give a combinatorial description of the maximal poset $\mF \cP(s_\cP,t_\cP)$ of a pasting scheme $\cP$. We note however that, given any two vertices $x$ and $y$ of $\cP$, by \cref{cor:sub-pasting-scheme} the hom-poset $\mF \cP(x,y)$ can always be realized as the maximal hom-poset
\[\mF \cP(x, y)\cong\mF \cP_{x,y}(s_{\cP_{x,y}}, t_{\cP_{x,y}})\]
 of the sub pasting scheme $\cP_{x,y}$.

\begin{thm}\label{thm:pasting-poset} Suppose a pasting scheme $\cP$ has atomic $2$-cells $\gamma_1,\ldots, \gamma_n$. The maximal hom-category $\mF \cP(s_\cP,t_\cP)$ is isomorphic to the full subposet $P^{\geq}$ of the cube $\{0 \leq 1\}^{\{\gamma_1,\ldots, \gamma_n\}}$ spanned by the points that belong to the regions satisfying the coordinate relations $\gamma_i \geq \gamma_j$ whenever $\gamma_i$ lies directly above $\gamma_j$.\footnote{Note that $P^{\geq}$ could be defined equally by the transitive closure of the ``lies directly above'' relation of \cref{lies directly above}. If a point in 
$\{0 \leq 1\}^n$ belongs to the region $\gamma_i \geq \gamma_j$ and also to the region $\gamma_j \geq \gamma_k$ then the point also belongs to the region $\gamma_i \geq \gamma_k$.}
\end{thm}

We give a few examples to explain the idea. In a pasting scheme $\cP$, the source path $\dom_\cP$ and target path $\cod_\cP$ from $s_\cP$ to $t_\cP$ bound a region in the plane, which contains the atomic 2-cells. The idea of this theorem is that any path from $s_\cP$ to $t_\cP$ in $\cP$ partitions this bounded region into the cells lying above that path and the cells lying below that path. The corresponding vertex in $\{0 \leq 1\}^{\{\gamma_1,\ldots, \gamma_n\}}$ is exactly the one that sets $\gamma_i = 1$ if $\gamma_i$ is above the path and $\gamma_i = 0$ if $\gamma_i$ is below the path. 

\begin{ex}
For example, the maximal hom-posets associated to the following pasting schemes are displayed below, with the $\alpha$ axis pointing left, the $\beta$ axis pointing out, and the $\gamma$ axis pointing down:
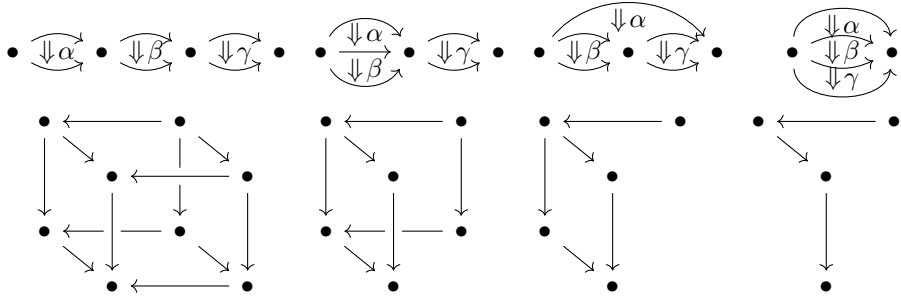
\setlength{\tabcolsep}{1pt}
\begin{center}
\begin{tabular}{cccc}
\begin{tikzcd}[column sep=2em]
\bullet \arrow[r, bend left] \arrow[r, bend right] \arrow[r, phantom, "\Downarrow\!\alpha"] & \bullet \arrow[r, bend left] \arrow[r, bend right]  \arrow[r, phantom, "\Downarrow\!\beta"]& \bullet \arrow[r, bend left] \arrow[r, bend right] \arrow[r, phantom, "\Downarrow\!\gamma"] & \bullet
\end{tikzcd} & 
\begin{tikzcd}[column sep=2em]
 \bullet  \arrow[r, "\displaystyle\Downarrow\!\alpha", "\displaystyle\Downarrow\!\beta"'] \arrow[r, bend left=60] \arrow[r, bend right=60]
 & \bullet \arrow[r, bend left] \arrow[r, bend right] \arrow[r, phantom, "\Downarrow\!\gamma"] & \bullet
\end{tikzcd} & 
\begin{tikzcd}[column sep=2em] \bullet \arrow[r, phantom, "\Downarrow\!\beta"] \arrow[rr, bend left=50, "\displaystyle\Downarrow\!\alpha"'] \arrow[r, bend right] \arrow[r, bend left] & \bullet \arrow[r, bend left] \arrow[r, bend right]\arrow[r, phantom, "\Downarrow\!\gamma"]& \bullet
\end{tikzcd} & 
\begin{tikzcd} 
 \bullet \arrow[r, bend left=25] \arrow[r, bend right=25]  \arrow[r, bend right=85, "\displaystyle\Downarrow\!\gamma"] \arrow[r, bend left=85, "\displaystyle\Downarrow\!\alpha"'] \arrow[r, phantom, "\Downarrow\!\beta"] & \bullet
 \end{tikzcd} \\
 \begin{tikzcd}[sep=small]
\bullet \arrow[dr] & & \bullet \arrow[dd] \arrow[ll]\arrow[dr] \\ & \bullet & & \bullet \arrow[ll, crossing over] \\ \bullet \arrow[from=uu] \arrow[dr]  & &\arrow[ll] \bullet \arrow[dr] \\ & \bullet \arrow[from=uu, crossing over] & & \bullet \arrow[ll] \arrow[from=uu]
\end{tikzcd} & 
\begin{tikzcd}[sep=small]
\bullet \arrow[dr] & & \bullet \arrow[dd] \arrow[ll] \\ & \bullet & &\\ \bullet \arrow[from=uu] \arrow[dr]  & &\arrow[ll] \bullet \\ & \bullet \arrow[from=uu, crossing over] & & 
\end{tikzcd} & 
\begin{tikzcd}[sep=small]
\bullet \arrow[dr] & & \arrow[ll] \bullet \\ & \bullet & &\\ \bullet \arrow[from=uu] \arrow[dr]  & & \\ & \bullet \arrow[from=uu, crossing over] & & 
\end{tikzcd} & 
\begin{tikzcd}[sep=small]
\bullet \arrow[dr] & & \arrow[ll] \bullet \\ & \bullet & &\\ {\phantom{\bullet}} & & {\phantom{\bullet}} \\ & \bullet \arrow[from=uu, crossing over] & & 
\end{tikzcd} 
\end{tabular}
\end{center}
\end{ex}

To prove this theorem, we use the following device, which allows us to ignore several special cases: if $\cP$ is a pasting scheme, then the \emph{augmentation} of $\cP$, denoted $\cP_+$, simply adds two new edges from $s_\cP$ to $t_\cP$, one of which lies above $\dom_\cP$ and the other of which lies below $\cod_\cP$.
This is again a pasting scheme, and the sub pasting scheme $\dom_\cP / \cod_\cP$ recovers $\cP$.
If $\{ \gamma_1, \dots, \gamma_n \}$ is the set of atomic 2-cells of $\cP$, we let $\gamma_0$ and $\gamma_{n+1}$ be the new 2-cells of $\cP_+$, with $\gamma_{n+1}$ to the right of the edge $\dom_{\cP_+}$ and $\gamma_0$ to the left of the edge $\cod_{\cP_+}$.

A point in $P^{\geq} \subset \{0\leq 1\}^{\{ \gamma_1, \dots, \gamma_n \}}$ can be identified with a function \[f\colon \{ \gamma_1, \dots, \gamma_n \} \to \{0\leq 1\}\] so that $f(\gamma_i) \geq f(\gamma_j)$ whenever $\gamma_i$ lies directly above $\gamma_j$. When convenient, we extend such functions to define corresponding functions \[f_+ \colon \{ \gamma_0, \gamma_1, \dots, \gamma_n, \gamma_{n+1} \} \to \{0 \leq 1\}\]  where we additionally insist that $f_+(\gamma_{n+1}) = 1$ and $f_+(\gamma_0) = 0$. Restriction of such an $f_+$ gives an isomorphism between the poset of such functions on the atomic 2-cells of $\cP_+$ and the poset $P^{\geq}$.

We prove \cref{thm:pasting-poset} by constructing explicit inverse isomorphisms that identify $\mF \cP(s_\cP,t_\cP)$ with $P^{\geq} \subset \{0 \leq 1\}^{\{\gamma_1,\ldots, \gamma_n\}}$. The easier to define is the \textbf{coordinatization} function, which turns a path $p$ from $s_\cP$ to $t_\cP$ in $\cP$ into a function $f \in P^{\geq}$, which we define by specifying its augmentation $f_+$. Within $\cP_+$, the path $p$ intersects the special edge $\dom_{\cP_+}$ only at endpoints, and we write $C_1$ for the bounded region of $\CC \backslash (p \cup \dom_{\cP_+})$.
Likewise we write $C_0$ for the bounded region of $\CC \backslash (p \cup \cod_{\cP_+})$.
If $\gamma$ is an atomic 2-cell of $\cP_+$, we say $f_+(\gamma) = i$ just when $\gamma$ is contained in the region $C_i$.

\begin{lem}
The coordinatization defines a function $\mF \cP(s_\cP,t_\cP) \to P^{\geq}$.
\end{lem}
\begin{proof}
We must show that the coordinatization $f$ of any path $p$ from $s_\cP$ to $t_\cP$ in $\cP$ satisfies $f(\gamma_i) \geq f(\gamma_j)$ whenever $\gamma_i$ lies directly above $\gamma_j$, in other words that we cannot have $\gamma_i$ in $C_0$ but $\gamma_j$ in $C_1$. In this case, the intersection of their boundaries lies then in the intersection of the closures of $C_0$ and $C_1$, which is exactly $p$. Thus, any edge $e$ witnessing that $\gamma_i$ lies directly above $\gamma_j$ needs to belong to $p$. In particular, $\gamma_i$ is to  the left of $e$ and $\gamma_j$ is to the right, which means that $\gamma_i$ lies in $C_1$ and $\gamma_j$ lies in $C_0$, a contradiction.
\end{proof}

The inverse mapping, called \textbf{pathification}, requires more work:

\begin{defn} 
Each function $f\colon \{ \gamma_1, \dots, \gamma_n \} \to \{0\leq 1\}$ with $f(\gamma_i) \geq f(\gamma_j)$ when $\gamma_i$ lies directly above $\gamma_j$ defines a coloring on the set of edges of $\cP$.
\begin{itemize}
\item Given an edge $e\in \cP$, write $\ell(e) \in \{ \gamma_0, \gamma_1, \dots, \gamma_n, \gamma_{n+1} \}$ for the 2-cell of $\cP_+$ to the left of $e$, and $r(e)$ for the 2-cell of $\cP_+$ to the right of $e$.
\item An edge $e$ of $\cP$ is colored \textbf{silver} if $f_+(\ell(e)) = f_+(r(e))$.
\item An edge $e$ of $\cP$ is colored \textbf{gold} if $f_+(\ell(e)) > f_+(r(e))$.
\end{itemize}
Note that by the restriction on $f$, all edges of $\cP$ are colored either silver or gold; we are not coloring $\dom_{\cP_+}$ or $\cod_{\cP_+}$.
\end{defn}

Using the ordering from \cref{cvn:inp-out-order}, 
we name the inner faces of $\cP_+$ adjacent to a vertex $v$ of $\cP$ as follows:
If $\inp_\cP(v) = \{ e_1' \leq \dots \leq e_m' \}$, we write $\lambda_{i-1}^v = r(e_i')$ and $\lambda_i^v = \ell(e_i')$. 
If $\out_\cP(v) = \{ e_1 \leq \dots \leq e_k \}$, we write $\rho_{j-1}^v = r(e_j)$ and $\rho_j^v = \ell(e_j)$.
These are consistent labelings. 
As an illustration, if $v \notin \{s_\cP, t_\cP\}$ then we have a picture like the following:
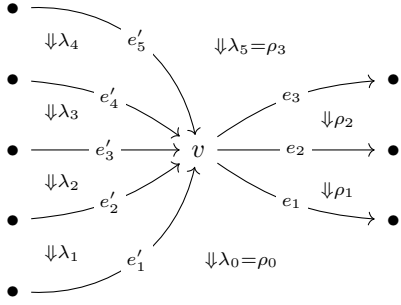
\[ \begin{tikzcd}[row sep=1.5em]
\bullet \ar[ddrr, bend left=40, "e_5'" description] & \arrow[dl, phantom, "\scriptstyle\Downarrow\lambda_4"]
& \arrow[dr, phantom, "\scriptstyle\Downarrow\lambda_5=\rho_3"]
\\
\bullet \arrow[drr, bend left=15, "e_4'" description] &  \arrow[dl, phantom, "\scriptstyle\Downarrow\lambda_3"] & & ~ &\bullet \\
\bullet \arrow[rr, "e_3'" description] & \arrow[dl, phantom, "\scriptstyle\Downarrow\lambda_2"] & v \ar[rr, "e_2" description] \ar[urr, bend left=15, "e_3" description] \ar[drr, bend right=15, "e_1" description] & \arrow[ur, phantom, "\scriptstyle\Downarrow\rho_2"] & \bullet  \\
\bullet  \arrow[rru, bend right=15, "e_2'" description] & \arrow[dl, phantom, "\scriptstyle\Downarrow\lambda_1"] & &  \arrow[ur, phantom, "\scriptstyle\Downarrow\rho_1"] ~ & \bullet \\
\bullet \ar[uurr, bend right=40, "e_1'" description] & & \arrow[ur, phantom, "\scriptstyle\Downarrow\lambda_0=\rho_0"]
\end{tikzcd} 
\]
The picture is similar for $s_\cP$ except there are no $\lambda_i$, and likewise $t_\cP$ has no $\rho_j$.

\begin{lem}
At most one input to a vertex is gold. At most one output of a vertex is gold.
\end{lem}
\begin{proof}
Suppose $\inp_\cP(v) = \{ e_1' \leq \dots \leq e_m' \}$ with $m > 1$. We have $1\geq f_+(\lambda_m) \geq f_+(\lambda_{m-1}) \geq \cdots \geq f_+(\lambda_0) \geq 0$, so there is at most one index $i$ with $f_+(\ell(e_i')) = f_+(\lambda_i) > f_+(\lambda_{i-1}) = f_+(r(e_i'))$.
\end{proof}

\begin{lem}
The vertex $s_\cP$ has a gold outgoing edge and the vertex $t_\cP$ has a gold incoming edge.
If $v\notin \{s_\cP, t_\cP\}$, then $v$ has a gold incoming edge if and only if it has a gold outgoing edge.
\end{lem}
\begin{proof}
If $v$ is an arbitrary vertex with $m$ incoming edges in $\cP$ and $k$ outgoing edges in $\cP$, then $v$ has an incoming gold edge if and only if $f_+(\lambda_m^v) = 1$ and $f_+(\lambda_0^v) = 0$, and has an outgoing gold edge if and only if $f_+(\rho_k^v) = 1$ and $f_+(\rho_0^v) = 0$.

If $v \notin \{ s_\cP, t_\cP\}$, then $\lambda_m^v = \rho_k^v$ and $\lambda_0^v = \rho_0^v$, so the second result follows.
If $v = s_\cP$, then $\rho_0^v = \gamma_0$ and $\rho_k^v = \gamma_{n+1}$, but our conditions on $f_+$ insist that $f_+(\gamma_0) = 0$ and $f_+(\gamma_{n+1})=1$. The $v = t_\cP$ case is similar.
\end{proof}

As $\cP$ has no directed cycles, the previous two lemmas give the following:

\begin{cor}
If $f \in P^{\geq}$, then the gold edges associated to $f$ form a path in $\cP$ from $s_\cP$ to $t_\cP$. \qed
\end{cor}

\begin{lem}\label{lem pathification injective} The pathification function is injective.
\end{lem}
\begin{proof}
Suppose $f$ and $g$ are distinct elements in $P^{\geq}$. In that case we may find an atomic 2-cell  $\gamma_i$ of $\cP$, with $f(\gamma_i) \neq g(\gamma_i)$. Since there are only finitely many atomic 2-cells, we may choose $\gamma_i$ to be a minimal element in the ``lies above'' partial order from \cref{transitive lies above} among those with this property, so that in particular any cell directly below $\gamma_i$ does not have this property.

Without loss of generality suppose $f(\gamma_i) = 0$ and $g(\gamma_i) = 1$.
Let $e$ be an edge of $\cod_{\gamma_i}$, and let $\gamma_j$ be the atomic 2-cell of $\cP_+$ lying to the right of $e$, which is directly below $\gamma_i$.
Then $f_+(\gamma_j) = 0 = g_+(\gamma_j)$, 
so $e$ is in the pathification of $g$ but not in the pathification of $f$.
\end{proof}

\begin{proof}[Proof of \cref{thm:pasting-poset}]
We will show that coordinatization and pathification are inverses, and also that they each preserve the partial orders (where the partial order on $\mF \cP(s_\cP, t_\cP)$ is from \cref{cor:hom-poset}).

If $p$ is a path from $s_\cP$ to $t_\cP$, we have that an edge $e$ of $\cP$ is either in $C_0$, $C_1$ or $\overline{C_0} \cap \overline{C_1} = p$, so the coordinatization of $p$ lies in $P^{\geq}$.
It is immediate that if $f$ is the coordinatization of a path $p$, then every edge in $p$ is gold, so coordinatization is a section of pathification. Since pathification is injective by \cref{lem pathification injective}, these functions are inverses.


Thus, the pathification and coordinatization functions define a bijection between the vertices of the posets $\mF \cP(x,y)$ and $P^{\geq}$. To see that these define an isomorphism of posets, we must argue that both constructions respect the partial orders, of a path $p$ being above a path $q$ in the case of $\mF \cP(x,y)$ and of a function $f$ being coordinatewise less than or equal to a function $g$ in $P^{\geq}$.

To see that coordinatization respects the ordering recall, by \cref{cor:hom-poset}, there is an arrow from $p$ to $q$ if and only if the path $p$ lies above the path $q$. In this case the coordinatization of $p$ clearly is less than or equal to that of $q$, for these coordinatizations will only differ on the coordinates corresponding to atomic 2-cells between $p$ and $q$, and the coordinatization of $p$ will assign these the value 0 while the coordinatization of $q$ will assign them the value $1$. 

Finally, we show that the pathification is order-preserving. 
If we have $f$ and $f'$ with $f \leq f'$ then the gold path of $f$ is above the gold path of $f'$.
This is true locally at a vertex $v$ which is in the intersection of two gold paths, since the 2-cells sent by $f$ to 1 are a subset of the 2-cells sent by $f'$ to 1.
But ``above'' for paths was a local condition at the vertices in their intersection.
\end{proof}

Combining \cref{thm:pasting-poset} with  \cref{cor:sub-pasting-scheme} we see that all hom-posets $\mF \cP(x,y)$ in the free 2-category generated by a pasting scheme are full sub-posets of hypercubes. Indeed, such hom-posets are necessarily sublattices:

\begin{lem}\label{lem:hom-lattice} For any pasting scheme $\cP$ and vertices $x,y \in \cP$, the hom-poset $\mF \cP(x,y)$ is a sublattice of a $\{0 \leq 1\}^k$, with finite meets and joins, where $k$ is the number of atomic $2$-cells of $\cP$ whose source and target vertices lie between $x$ and $y$.
\end{lem}
\begin{proof}
By \cref{thm:pasting-poset} we may identify the vertices of $\mF \cP(x,y)$ with their coordinatizations. As the inclusion $\mF \cP(x,y) \hookrightarrow \{0 \leq 1\}^n$ is full, it suffices to prove that the meet and join of any pair of points in the subposet also lie in the subposet. Suppose $f$ and $g$ are two such coordinatizations. This means that $f, g \in \{0 \leq 1\}^n$ are sequences of vertices that obey whatever relations the coordinatewise relations imposed by the atomic 2-cells of $\cP$. Our task is to show that $f\wedge g$ and $f \vee g \in \{0 \leq 1\}^n$ also satisfy the restrictions that define the subposet. But note that the meet and join share all of the common coordinates of $f$ and $g$, but then have all 0s, in the case of the meet, or all 1s, in the case of the join, in the coordinates where they disagree. If it is required that the $i$th coordinate is less than or equal to the $j$th coordinate and both coordinates are among the vertices that have been changed, there is no problem, since in both the meet and the join the values at these coordinates are equal. But otherwise the configuration in question appears in either $f$ or $g$, and since these vertices are coordinatizations, the condition must be  allowed in the subposet.
\end{proof}

\begin{rmk} It is natural to wonder whether there is a characterization of the sublattices of hypercubes that arise as hom-categories for some pasting scheme. There is one obvious further relation on the defining coordinate relations: by \cref{transitive lies above} we cannot have both $\alpha \geq \beta$ and $\beta \geq \alpha$ and, more generally, we cannot have any cycles $\gamma_1 \geq \gamma_2$, \ldots, $\gamma_{k-1}\geq \gamma_k$, $\gamma_k \geq \gamma_1$. But we imagine there may well be other restrictions imposed by the  impossibility of embedding certain configurations of atomic 2-cells into the plane. We leave this as an open question for future study.
\end{rmk}

\section{Dwyer maps and their nerves}\label{sec:dwyer}

In a sense that will be made precise in \S\ref{sec:uniqueness}, the difference between a pasting diagram in an $(\infty,2)$-category and the homotopy coherent diagram it generates boils down to the distinction between a pushout of nerves of categories and the nerve of the category defined by the pushout. Our aim in this section is to prove Corollary \ref{cor:hom-dwyer-map}, which will be applied in a crucial way in the inductive step in the proof of our main theorem.

We deduce this result as a special case of a theorem of independent interest that is tailored to exactly this sort of situation. Its statement concerns a class of functors first considered by Thomason under the name ``Dwyer maps,'' which are introduced in \S\ref{ssec:dwyer}, where we review the literature and observe that various functors related to the hom-categories of pasting schemes define Dwyer maps.

In \S\ref{ssec:dwyer-pushout}, we state Theorem \ref{AnodyneDwyer}, which is proven in a companion paper \cite{HORR-Dwyer}. This result demonstrates that the canonical comparison between the pushout of nerves of categories and the nerve of the pushout is a weak categorical equivalence, provided that one of the functors in the span is a Dwyer map. 
In fact, when the other functor in the span is injective on objects and faithful, as it is in our case of interest, the comparison map is inner anodyne.
Corollary \ref{cor:hom-dwyer-map} follows immediately.

\subsection{Dwyer maps}\label{ssec:dwyer}

Thomason refers to certain full inclusions of 1-categories as Dwyer maps \cite[Definition 4.1]{ThomasonModelCat}. These feature in a central way in the construction of the Thomason model structure on categories.

\begin{defn}[Thomason]\label{defn:Dwyer-map}
A full sub-$1$-category inclusion $I \colon \cA \hookrightarrow \cB$ is \textbf{Dwyer map}
if the following conditions hold.
\begin{enumerate}[label=(\roman*)]
\item The category $\cA$ is a \emph{sieve} in $\cB$, meaning there is a necessarily unique functor $\chi \colon \cB \to \cattwo$ with $\chi^{-1}(0) = \cA$. We write $\cV:=\chi^{-1}(1)$ for the complementary \emph{cosieve} of $\cA$ in $\cB$. 
\item The inclusion $I \colon \cA \hookrightarrow\cW$ into the \emph{minimal cosieve}\footnote{Explicitly $\cW$ is the full subcategory of $\cB$ containing every object that arises as the codomain of an arrow with domain in $\cA$.} $\cW \subset \cB$ containing $\cA$ admits a right adjoint left inverse $R\colon \cW\to \cA$, a right adjoint for which the unit is an identity.
\end{enumerate}
\end{defn}

Schwede describes Dwyer maps as ``categorical analogs of the inclusion of a neighborhood deformation retract'' \cite{SchwedeOrbispaces}. In fact all of the examples of Dwyer maps considered in this paper are more like deformation retracts, in that the cosieve $\cW$ generated by $\cA$ is the full codomain category $\cB$.

\begin{ex}
\label{basicDwyer}
The vertex inclusion $0 \colon \catone \to \cattwo$ is a Dwyer map, with $! \colon \cattwo \to \catone$ the right adjoint left inverse.
The other vertex inclusion $1 \colon \catone \to \cattwo$ is not a Dwyer map.
\end{ex}

Generalizing the previous example:

\begin{ex}\label{ex:new-terminal-Dwyer} If $\cA$ is a category with a terminal object and  $\cA^{\triangleright}$ is the category which formally adds a new terminal object, then the inclusion $\cA \hookrightarrow \cA^{\triangleright}$ is a  Dwyer map.\footnote{If $\cA$ does not have a terminal object, then $\cA \to \cA^{\triangleright}$ need not be a Dwyer map. Indeed, if $\cA=\catone\amalg\catone$, the only cosieve containing $\cA$ is $\cA^\triangleright$ itself, and there cannot be a right adjoint $\cA^\triangleright \to \cA$ as $\cA$ does not have a terminal object.}
\end{ex}

We warn the reader that we are using the original notion of Dwyer map, not the pseudo-Dwyer maps introduced by Cisinski \cite{CisinskiDwyer}, which are retracts of Dwyer maps. In particular, our Dwyer maps are not closed under retracts. They do, however, enjoy the following closure properties, as we now recall:

\begin{lem}[{\cite[Proposition 1.1]{SchwedeOrbispaces}}]\label{lem:dwyer_product} Any product $\cC \times I \colon \cC \times \cA \hookrightarrow \cC \times \cB$ of a Dwyer map $I$ with a category $\cC$ remains a Dwyer map.
\label{productDwyer}
\end{lem}

\begin{lem}[{\cite[Proposition 4.3]{ThomasonModelCat}}]
\label{pushoutDwyer}
Any pushout of a Dwyer map $I$ defines a Dwyer map $J$:
\[
\begin{tikzcd}
  \cA \arrow[d, "I"', hook]\arrow[r, "F"]\arrow[dr, phantom, "\ulcorner" very near end] &\cC \ar[d, "J", hook]\\
  \cB \arrow[r, "G" swap] &\cD.
\end{tikzcd}
\]
\end{lem}

Note for example, that \cref{pushoutDwyer} explains the Dwyer map of \cref{ex:new-terminal-Dwyer}: if $\cA$ has a terminal object $t$, then the pushout
\[
\begin{tikzcd} \arrow[dr, phantom, "\ulcorner" very near end] \catone \arrow[d, hook, "0"'] \arrow[r, "t"] & \cA \arrow[d, hook] \\ \cattwo \arrow[r] & \cA^\triangleright
\end{tikzcd}
\]
defines the category $\cA^\triangleright$.

We now give an example of this example. By \cref{lem:hom-lattice}, the hom-posets $\mF\cP(x,y)$ of any pasting scheme $\cP$ are lattices, and in particular, have a terminal object, namely the bottommost path $q$ from $x$ to $y$ identified by \cref{lem:maximal-path-x-to-y}. Suppose further that $q \subset\cod_\cP$ is a subpath of the codomain path of $\cP$ and that $x \prec y$. Then we can form a new pasting scheme $\cP\cup\alpha$ by attaching a new atomic 2-cell from $x$ to $y$ below $\cP$ by identifying $\dom_\alpha$ with $q$, a process that we summarize by saying that we are ``attaching $\alpha$ at the bottom of $\cP$ from $x$ to $y$.''
The following lemma describes the effect of attaching a 2-cell on the bottom on the hom-poset from $x$ to $y$.

\begin{lem}\label{lem:cell-at-the-bottom}
Let $\cP \cup \alpha$ be a pasting scheme obtained by attaching a $2$-cell $\alpha$ at the bottom of a pasting scheme $\cP$ from $x$ to $y$. Then we have a pushout of hom-posets
\[
\begin{tikzcd} \catone \arrow[d, hook, "0"'] \arrow[r, "\dom_\alpha"] \arrow[dr, phantom, "\ulcorner" very near end] &\mF{\cP}(x,y) \arrow[d, hook] \\ \cattwo \arrow[r] & \mF(\cP \cup \alpha)(x,y)
\end{tikzcd}
\]
and hence $\mF{\cP}(x,y) \hookrightarrow \mF(\cP \cup \alpha)(x,y)$ is a Dwyer map.
\end{lem}

\begin{proof}
The poset $\mF{\cP}(x,y)$ has a terminal object, namely $\cod _{\cP_{x,y}} = \dom_\alpha$.
By \cref{cor:hom-poset}, $\mF{\cP}(x,y)^\triangleright \cong \mF{(\cP\cup \alpha)}(x,y)$, and the result follows.
\end{proof}

\begin{rmk}
\label{posetandDwyer}
Ordinarily we would have to be careful to disambiguate between pushouts taken in the category of categories and pushouts taken in the category of posets, but as observed by Raptis in \cite[Lemma 2.5]{Raptis:HTP}, the inclusion of posets into categories preserves pushouts along Dwyer maps. This applies in particular to the pushouts considered in  \cref{lem:cell-at-the-bottom} and \cref{prop:composition-pushout} and implies that the pushouts of categories considered here are again posets.
\end{rmk}

In fact, with a little more care, we can describe the effect on all of the hom-posets of $\cP$ by the procedure of attaching a new atomic 2-cell on the bottom.

\begin{prop}\label{prop:composition-pushout} Consider the pasting scheme $\cP \cup \alpha$ with an atomic $2$-cell $\alpha$ attached at the bottom of $\cP$ from $x$ to $y$, and let $a$ and $z$ be objects in the sub pasting scheme $\cP$. Then there is a pushout of hom-posets
\[
\begin{tikzcd} \mF \cP(y,z) \times \mF \cP(x,y) \times \mF \cP(a,x) \arrow[r, "\circ", hook] \arrow[d, hook] \arrow[dr, phantom, "\ulcorner" very near end] & \mF \cP(a,z) \arrow[d, hook] \\ \mF \cP(y,z) \times \mF (\cP \cup \alpha)(x,y) \times \mF \cP(a,x) \arrow[r, hook] & \mF (\cP \cup \alpha)(a,z)
\end{tikzcd}
\]
and hence $\mF \cP (a,z) \hookrightarrow \mF (\cP \cup \alpha)(a,z)$ is a Dwyer map.
\end{prop}

\begin{proof}
Note that if $a$ and $z$ are objects in the pasting scheme $\cP$ with either $a \not\preccurlyeq x$ or $y \not\preccurlyeq z$ (or both), then the statement says that $\mF \cP(a,z) = \mF(\cP \cup \alpha)(a,z)$ since the posets on the left are empty.
But this equality holds, as any path in $\cP\cup \alpha$ from $a$ to $z$ will necessarily lie in $\cP$. For the remainder of the proof, we assume $a \peq x \prec y \peq z$.

Denote by $\cD$ the pushout of categories
\[
\begin{tikzcd} \mF \cP(y,z) \times \mF \cP(x,y) \times \mF \cP(a,x) \arrow[r, "\circ", hook] \arrow[d, hook] \arrow[dr, phantom, "\ulcorner" very near end] & \mF \cP(a,z) \arrow[d, hook] \\ \mF \cP(y,z) \times \mF (\cP \cup \alpha)(x,y) \times \mF \cP(a,x) \arrow[r, hook] & \cD.
\end{tikzcd}
\]
We will show that the comparison map
$\cD\to\mF (\cP \cup \alpha)(a,z)$
is an isomorphism of categories. As a preliminary observation, we note that the two maps in the underlying span are fully-faithful inclusions of posets: the vertical one is fully faithful (and a Dwyer map) by \cref{lem:dwyer_product,lem:cell-at-the-bottom} and the horizontal one is fully faithful by \cref{prop:concat_ff}. 

 First, we show that the comparison functor is bijective on objects.
   The objects of $\mF (\cP\cup\alpha)(a,z)$ are
paths in $\cP \cup \alpha$ from $a$ to $z$, which satisfy exactly one of the following conditions:
\begin{enumerate}[label=(\roman*)]
\item the path passes through the codomain of $\alpha$;
\item the path lies in $\cP$ and passes through $x$ and $y$;
\item the path lies in $\cP$ but does not pass through both $x$ and $y$.
\end{enumerate}
The three possibilities correspond respectively to
\begin{enumerate}[label=(\roman*)]
\item the objects in the complement of the image of the inclusion \[\mF \cP(y,z) \times \mF \cP(x,y) \times \mF \cP(a,x)\to\mF \cP(y,z)\times \mF(P\cup\alpha)(x,y) \times \mF \cP(a,x);\]
\item the objects of $\mF \cP(y,z) \times \mF \cP(x,y) \times \mF \cP(a,x)$;
\item the objects in the complement of the image of the inclusion \[\mF \cP(y,z) \times \mF \cP(x,y) \times \mF \cP(a,x)\to\mF \cP(a,z).\]
\end{enumerate}
It follows that the comparison functor $\cD\to\mF (P\cup\alpha)(a,z)$ is bijective on objects, as claimed.
    
By \cref{productDwyer,,lem:cell-at-the-bottom,pushoutDwyer}, $\mF \cP(a,z)\hookrightarrow \cD$ is a Dwyer map and $\cD$ is a poset by \cref{posetandDwyer}. Hence, the comparison functor $\cD\to\mF (P\cup\alpha)(a,z)$  is automatically faithful. 

To complete the proof that this comparison in fact defines an isomorphism, it suffices to argue that it is full.
For this, we need to analyze the different relations we can have between $p$ and $q$ in $\mF(\cP \cup \alpha)(a,z)$ and show the same relations hold in $\cD$. 
Both $\mF \cP(y,z) \times \mF (\cP \cup \alpha)(x,y) \times \mF \cP(a,x)$ and $\mF \cP(a,z)$ are full subcategories of  $\mF (\cP\cup \alpha)(a,z)$ (by \cref{cor:loc_ff_sub} and \cref{prop:concat_ff}), and since  fully-faithful inclusions are closed under pushouts \cite[Proposition 5.2]{FritschLatch:HomotopyInversesForNerve} they are also full subcategories in $\cD$. So we only need to deal with the case of $p$ belonging to $\mF \cP(a,z)$  and $q$ belonging to $\mF \cP(y,z) \times \mF (\cP \cup \alpha)(x,y) \times \mF \cP(a,x) $, with neither belong to $\mF \cP(y,z) \times \mF \cP(x,y) \times \mF \cP(a,x)$.
In particular $\cod_\alpha$ must be a subpath of $q$, so we may define a path $q'$ by replacing $\cod_\alpha$ with $\dom_\alpha$ in $q$.
    
By \cref{cor:hom-poset}, there is an arrow from $p$ to $q$ in $\mF (P \cup \alpha)(a,z)$ if and only if $p$ lies above $q$. If this is the case, then $p$ lies above $q'$ in  $\mF \cP(a,z)\subset\cD$ and $q'$ lies above $q$ in $\mF \cP(y,z) \times \mF (\cP \cup \alpha)(x,y) \times \mF \cP(a,x) \subset\cD$, so there is a morphism from $p$ to $q$ in $\cD$, as desired.

Again by  \cref{cor:hom-poset}, $\mF(\cP \cup \alpha)(a,z)$ does not contain any arrows from $q$ to $p$ because it is not possible for $q$ to lie above $p$. 
Indeed, we know that $q'$ lies above $q$, so if $q$ were to lie above $p$ then $q'$ would also be above $p$. In this case, $p$ would necessarily contain $\dom_\alpha \subset \cod_{\cP}$ as a subpath, and in particular $p$ would pass through $x$ and $y$, which contradicts the hypothesis that $p$ does not belong to $\mF \cP(y,z) \times \mF \cP(x,y) \times \mF \cP(a,x)$. 
\end{proof}

\subsection{Nerves, pushouts, and weak categorical equivalence}\label{ssec:dwyer-pushout}

In this section, we return to the setting of a pushout of a Dwyer map considered in \cref{pushoutDwyer}. 
In a companion paper \cite{HORR-Dwyer}, we show that such pushouts of categories can also be regarded as pushouts of $(\infty,1)$-categories in the sense made precise by the following result, which considers the nerve embedding from categories into quasi-categories:
 
 \begin{thm}[{\cite[1.5, 4.1]{HORR-Dwyer}}] 
    \label{AnodyneDwyer}
    Let 
    \[
    \begin{tikzcd}
      \cA \arrow[d, "I"', hook]\arrow[r, "F"]\arrow[dr, phantom, "\ulcorner" very near end] &\cC \ar[d, "J", hook]\\
      \cB \arrow[r, "G" swap] &\cD
    \end{tikzcd}
    \]
    be a pushout of categories, and assume $I$ to be a Dwyer map. Then the induced map of simplicial sets 
    \begin{equation}\label{eq:Dwyer-pushout-map}
        N\cC\aamalg{N\cA} N\cB \to N\cD\end{equation} is a weak categorical equivalence, and is in fact inner anodyne in the case where $F$ is injective on objects and faithful. 
     \end{thm}
     
   Here weak categorical equivalences refer to the weak equivalences in Joyal's model structure for quasi-categories. 
  \cref{AnodyneDwyer} is a refinement of a similar result of Thomason \cite[Proposition 4.3]{ThomasonModelCat} which proves that the same map is a weak homotopy equivalence.\footnote{Recall that Quillen's model structure for Kan complexes, whose weak equivalences are the weak homotopy equivalences, is a left Bousfield localization of Joyal's model structure, with more weak equivalences and fewer fibrant objects.}

We briefly explain the idea of the proof of \cref{AnodyneDwyer} in the case of interest here, when the other functor in the span is injective on objects and faithful. In that setting, the comparison map \eqref{eq:Dwyer-pushout-map} can easily be seen to be an inclusion, using an explicit description of the objects and morphisms in the category $\cD$ defined by the pushout of a Dwyer map established in \cite[Proof of Lemma 2.5]{BMOOPY}; cf.\ also \cite[Construction 1.2]{SchwedeOrbispaces} and \cite[\textsection7.1]{AraMaltsiniotisVers}. 

It remains only to define a suitable filtration of the map \eqref{eq:Dwyer-pushout-map} as a composite of pushouts of inner horn inclusions. We do this firstly by dimension and then by ``bridge index,'' which is defined in terms of the canonical projection map $N\cD \to \Delta^1$. As is typical for such combinatorial arguments, we identify a certain subset of non-degenerate ``bascule'' simplices in $N\cD$ that we attach via inner horns. The justification for this construction requires a carefully defined bijection between bridging simplices and their ``bascule lifts,'' a canonical simplex of dimension one larger which has the bridging simplex as a particular face. See \cite[\S3-4]{HORR-Dwyer} for more details.

The following application of  \cref{AnodyneDwyer} will play a key role in the proof of our main theorem:

\begin{cor}\label{cor:hom-dwyer-map}
Consider the pasting scheme $\cP \cup \alpha$ with an atomic $2$-cell $\alpha$ attached to the bottom of a pasting scheme $\cP$ from $x$ to $y$. Then for any objects $a$ and $z$ in the sub pasting scheme $\cP$ the natural map

\[ N{\mF}\cP(y,z) \times \Delta^1 \times N{\mF}\cP(a,x) \aamalg{N{\mF}\cP(y,z) \times \Delta^0 \times N{\mF}\cP(a,x)}N{\mF}\cP(a,z) \longrightarrow N\mF(\cP \cup \alpha)(a,z)\] is an inner anodyne extension and in particular
a weak categorical equivalence.
\end{cor}
\begin{proof}
By \cref{lem:cell-at-the-bottom}  and \cref{prop:composition-pushout}, the hom-posets form a composite pushout square
\[
\begin{tikzcd}[column sep=large] \mF \cP(y,z) \times \catone \arrow[d, "0"', hook] \times \mF \cP(a,x) \arrow[r, "-\circ\dom_\alpha\circ-", hook] \arrow[dr, phantom, "\ulcorner" very near end] & \mF \cP(a,z) \arrow[d, hook] \\ \mF \cP(y,z) \times \cattwo \times \mF \cP(a,x) \arrow[r, hook] & \mF (\cP \cup \alpha)(a,z).
\end{tikzcd}
\]
By \cref{productDwyer} and \cref{pushoutDwyer} the vertical functors are all Dwyer maps. The result follows from \cref{AnodyneDwyer}.
\end{proof}

\section{The homotopical uniqueness of  \texorpdfstring{$(\infty,2)$}{(infinity,2)}-categorical pasting composition}\label{sec:uniqueness}

In this section, we prove our main theorem, establishing the homotopical uniqueness of pasting composition using simplicial categories to model $(\infty,2)$-categories. In \S\ref{ssec:simplicial-model}, we review the simplicial categories model and give a precise statement of our result, which appears as \cref{thm:uniquepasting}. The remaining sections are occupied by its proof, which proceeds by an induction along the lines permitted by \cref{cor:pasting-scheme-presentations}, where a pasting scheme is built inductively by attaching an atomic 2-cell ``at the bottom'' along its domain path. 

We find it most expedient to split the inductive step into two cases, considering first the special case in which the codomain of the attached 2-cell is defined by a single edge, rather than a path of edges. This case is treated in \S\ref{ssec:base-cases}, while the general inductive step, in which the codomain of the attached cell is allowed to be any path of edges, is deduced from this special case in \S\ref{ssec:general-inductive-step} after which \cref{thm:uniquepasting} is proven. A direct argument that treats both cases at once is certainly possible---and indeed one was written up in the preparation of this manuscript---but we felt that the technicalities in the direct construction obscured the main ideas of the argument, so we prefer to leave this as an exercise for the reader who suspects they might prefer the other approach.

Finally, in \S\ref{ssec:model-independence} we discuss how an analogous result can be proven in a generic model of $(\infty,2)$-categories. As a corollary of our main theorem, we also establish that the space of composites of a given pasting diagram in an $(\infty,2)$-category is contractible. 

\subsection{The simplicial category model of \texorpdfstring{$(\infty,2)$}{(infinity,2)}-categories}\label{ssec:simplicial-model}

To define the simplicial categories model of $(\infty,2)$-categories, we make use of a product-preserving functor $\tau \colon \sSet \to \Set$ first considered by Joyal, which takes a simplicial set to the set of isomorphism classes of objects in its homotopy category. Applying this functor hom-wise to a simplicially enriched category $\cC$, this defines an ordinary 1-category $\tau_*\cC$ that might be regarded as a homotopy category associated to $\cC$, using the notion of homotopy class of maps provided by the Joyal model structure on simplicial sets.

\begin{thm}[{the model structure for $(\infty,2)$-categories}]\label{def sCat model structure}
The category $\sCat$ of simplicially enriched categories admits a left-proper, combinatorial model structure in which:
\begin{itemize}
    \item A map $f \colon \cC\to \cD$ is a \textbf{weak equivalence} just when:
\begin{enumerate}[label=(W\arabic*)]
    \item For each pair of objects $x,y$, the map $\cC(x,y) \to \cD(fx,fy)$ is a weak categorical equivalence of simplicial sets, and \label{DK2 local equiv} 
    \item the functor $\tau_* f \colon \tau_* \cC \to \tau_* \cD$ is essentially surjective. \label{DK2 global equiv}
\end{enumerate}
\item An object $\cC$ is \textbf{fibrant} if and only if $\cC(x,y)$ is a quasi-category for every pair of objects $x,y$.
\item A map $f \colon \cC\to \cD$ whose \emph{codomain is fibrant} is a \textbf{fibration} just when:
\begin{enumerate}[label=(F\arabic*)]
    \item For each pair of objects $x,y$, the map $\cC(x,y) \to \cD(fx,fy)$ is a fibration in the Joyal model structure on $\sSet$, and \label{DK2 local fib}
    \item the functor $\tau_* f \colon \tau_* \cC \to \tau_* \cD$ is an isofibration.
    \label{DK2 global fib}
\end{enumerate}
\item 
The \textbf{cofibrations} are the cofibrations in the model structure for $(\infty,1)$-cat\-e\-gories on $\sCat$  \cite{Bergner:MS}.
\end{itemize}
\end{thm}

\begin{proof}
The existence of the model structure is a special case of \cite[Proposition A.3.2.4]{HTT}---using \cite[Definition A.3.2.1]{HTT} and \cite[Definition A.3.2.9]{HTT} for \ref{DK2 global equiv} and \ref{DK2 global fib}---while the characterization of the fibrations with fibrant codomain is \cite[Example A.3.2.23]{HTT} and \cite[Theorem A.3.2.24]{HTT}.
\end{proof}

Recall that simplicially enriched categories $\cA$ may be identified with simplicial objects in $\Cat$ comprised of categories $\cA_n$ for $n \geq 0$ with a common set of objects and identity-on-objects functors $\alpha^* \colon \cA_n \to \cA_m$ for each simplicial operator $\alpha \colon [m] \to [n]$ in $\DDelta$. 
From this vantage point, there is a useful characterization of the cofibrant objects in the model structure for $(\infty,2)$-categories due to Dwyer and Kan that have more recently been christened ``simplicial computads'' \cite[\textsection16.2]{Riehl:2014ch}.

\begin{defn}\label{defn:computad} A simplicial category, presented as a simplicial object $\cA_\bullet \in \Cat^{\DDelta^\op}$ is a \textbf{simplicial computad} just when:
\begin{itemize}
\item for each $n \geq 0$, the category $\cA_n$ of objects and $n$-arrows is free on a reflexive directed graph whose non-identity arrows are called \textbf{atomic $\boldsymbol{n}$-arrows}; and
\item the degeneracy functors $\sigma^* \colon \cA_n \to \cA_m$ indexed by epimorphisms $\sigma \colon [m] \twoheadrightarrow [n] \in \DDelta$ preserve atomic arrows.
\end{itemize}
\end{defn}

For example:

\begin{prop}\label{prop:pasting-computad} Let $\mF \cP$ be the free $2$-category generated by a pasting scheme and consider the simplicial category $N_* \mF \cP$ defined by applying the nerve hom-wise. Then $N_*\mF \cP$ is a simplicial computad whose
\begin{itemize}
    \item objects are the vertices of $\cP$,
    \item atomic $0$-arrows $e \colon x \to y$ are the edges of $\cP$,
    \item atomic $1$-arrows $(p_0 \geq p_1) \colon x \to y$ are given by parallel pairs of paths, with $p_0$ above $p_1$, so that the only vertices $p_0$ and $p_1$ have in common are their endpoints, and 
    \item atomic $n$-arrows $(p_0 \geq \cdots \geq p_n) \colon x \to y$ are given by parallel paths, with each $p_i$ above $p_{i+1}$,  so that the only vertices $p_0$ and $p_n$ have in common are their endpoints.
\end{itemize}
\end{prop}
\begin{proof}
For any 2-category $\mathcal{C}$, the simplicial category $N_*\mathcal{C}$ may be described as follows:
\begin{itemize}
\item its objects and 0-arrows are the 0-cells and 1-cells of $\mathcal{C}$;
\item its 1-arrows are the 2-cells of $\mathcal{C}$, considered as arrows from their 0-cell source to the 0-cell target;
\item its $n$-arrows are vertically composable sequences of $n$ 2-cells, considered as arrows from their common 0-cell source to their common 0-cell target.
\end{itemize} 
An $n$-arrow in the nerve of a 2-category is degenerate if and only if at least one of the constituent 2-cells is an identity on its boundary 1-cells. 

An $n$-arrow is atomic if and only if it cannot be factored horizontally.\footnote{Warning: the atomic 1-arrows include the \emph{atomic $2$-cells}, which cannot be expressed as a non-trivial pasting composite, but are considerably more general. A vertical composite of 2-cells might define an atomic 1-arrow if there is no way to factor the vertical composite horizontally.} 
Specializing \cref{defn:computad}, the simplicial category $N_*\mathcal{C}$ is a simplicial computad just when each $n$-arrow admits a unique factorization into atomic $n$-arrows. The second condition---that degenerate images of atomic $n$-arrows are atomic $n$-arrows---is automatic in the case of simplicial categories of the form $N_*\mathcal{C}$.

Now we specialize to the case of the 2-category $\mF{\cP}$. Recall from \cref{cor:hom-poset} that there is a 2-cell from $p$ to $q$ in $\mF{\cP}$ just when $p$ and $q$ are parallel paths of edges from $x$ to $y$ and $p$ is above $q$. Thus, an $n$-arrow in $N_*\mF \cP$ is given by a sequence of $n+1$ parallel paths of edges $p_0, \ldots, p_n$ with $p_i$ above $p_{i+1}$, these being the source and target paths in the sequence of vertically composable 2-cells. Such an $n$-arrow is degenerate if and only if $p_i = p_{i+1}$ for some $0 \leq i < n$.

We refer to an $n$-arrow in $N_*\mF \cP$ as \textbf{horizontally indecomposable} if there is no interior vertex $x  \prec u  \prec y$ that appears in every path in the sequence. Equivalently, this can be expressed by asking that the only vertices $p_0$ and $p_n$ have in common are their endpoints.
 
We claim that $N_*\mF \cP$ is a simplicial computad whose
\begin{itemize}
    \item objects are the vertices of $\cP$;
    \item atomic 0-arrows are the edges of $\cP$;
    \item atomic 1-arrows are the 2-cells $\alpha$ of $\mF \cP$ that are horizontally indecomposable, in the sense that the paths $\dom_\alpha$ and $\cod_\alpha$ do not share any vertices, aside from their common endpoints;
    \item atomic $n$-arrows are the vertically composable sequences of 2-cells of $\mF \cP$ that are horizontally indecomposable, in the sense that there is no vertex contained in the source and target paths of every  2-cell, aside from their common endpoints. 
\end{itemize}

Since degeneracies of horizontally indecomposable $n$-arrows are  horizontally indecomposable, it remains only to argue that a generic $n$-arrow factors uniquely as a composite of atomic $n$-arrows. 

We now argue that every $n$-arrow given by a sequence of $n+1$ parallel paths of edges $p_0 \geq \cdots \geq p_n$ from $x$ to $y$, with each $p_i$ above $p_{i+1}$,  factors uniquely as a composite of atomic $n$-arrows. The vertical composite 2-cell is the unique composite 2-cell in the sub pasting scheme $p_0/p_n$ of $\cP$ comprised of the cells that lie between the paths $p_0$ and $p_n$ (see \cref{def:p-over-q}), and we may restrict to this sub pasting scheme. By \cref{lem:paths_partition}, we may factor these paths as
\begin{align*}
p_0&= r_{0} p_{0,1} r_{1} p_{0,2} \cdots r_{m-1} p_{0,m} r_{m} \\
p_n&= r_{0} p_{n,1} r_{1} p_{n,2} \cdots r_{m-1} p_{n,m} r_{m}
\end{align*}
where
\begin{itemize}
\item the $r_{j}$ are paths, possibly of length zero (containing only a vertex), 
\item each path $p_{i,j}$  has positive length, and
\item $p_{0,j}$ and $p_{n,j}$ intersect at their endpoints and do not intersect at any interior vertices, and $p_{0,j}$ lies above $p_{n,j}$. 
\end{itemize}

The vertices $u_i$ appearing on both $p_0$ and $p_n$ are exactly the vertices that lie on the subpaths $r_j$ and hence are  linearly ordered $x = u_0 \prec u_1 \prec \cdots \prec u_{k-1} \prec u_{k} = y$; note $k \geq m$.

The paths $r_0, \ldots, r_m$ are also subpaths of each intermediate path $p_i$ from $x$ to $y$. Thus, we have similar factorizations
\[  
p_i = r_{0} p_{i,1} r_{1} p_{i,2} \cdots r_{m-1} p_{i,m} r_{m} 
\]
where each $p_{i,j}$ has positive length and lies above $p_{i+1,j}$ and below $p_{i-1,j}$; note the path $p_i$  might intersect with its  neighbors at additional interior vertices. The $n$-arrow $p_0 \geq\cdots\geq p_n$ then factors as a composite of the horizontally indecomposable $n$-arrows  from $u_{s}$ to $u_{s+1}$. These atomic factors have two forms: some are given by the parallel paths $p_{0,j} \geq \cdots \geq p_{n,j}$, while others are degenerate $n$-arrows on an atomic $0$-arrow $e$, arising as an edge $e \colon u_s \to u_{s+1}$ of one of the paths $r_j$. As this is evidently the unique factorization of the $n$-arrow $p_0 \geq\cdots\geq p_n$ into horizontally indecomposable $n$-arrows, $N_* \mF \cP$ is a simplicial computad, as claimed. 
\end{proof}

\begin{rmk}\label{rmk:sub-pasting-sub-computad} Extending \cref{prop:pasting-computad}, if $\cQ \subset \cP$ is a sub pasting scheme that is full on atomic 2-cells, then $N_*\mF\cQ \hookrightarrow N_*\mF\cP$ is a simplicial subcomputad inclusion, preserving atomic arrows in each dimension.
\end{rmk}

There is an adjunction
\[
\begin{tikzcd} \sSet \arrow[r, bend left, "{\cattwo[-]}"] \arrow[r, phantom, "\bot"] & \slicel{\catone+\catone}{\sCat} \arrow[l, bend left, "\hom"]
\end{tikzcd}
\]
that sends a simplicial set to its ``directed suspension,'' regarded as a simplicial category with two specified objects ``$-$'' and ``$+$'' and the homs:
\[
\cattwo[X](-,-) = \cattwo[X](+,+) \coloneqq \Delta^0,  \qquad \cattwo[X](-,+) \cong X, \qquad \cattwo[X](+,-) \coloneqq \varnothing.\]
This functor has a right adjoint that extracts the specified hom from a simplicial category with two specified objects.  

\begin{lem}\label{lem:directed-suspension-adjunction}
If $\sSet$ has the Joyal model structure for $(\infty,1)$-categories and $\sCat$ has the model structure for $(\infty,2)$-categories, then the adjunction
\[
\begin{tikzcd} \sSet \arrow[r, bend left, "{\cattwo[-]}"] \arrow[r, phantom, "\bot"] & \slicel{\catone+\catone}{\sCat} \arrow[l, bend left, "\hom"]
\end{tikzcd}
\]
is Quillen. 
\end{lem}

So in particular, the simplicial categories $\cattwo[X]$ are cofibrant and moreover $\cattwo[-]$ carries Joyal trivial cofibrations to trivial cofibrations.

\begin{proof}
By \cite[Proposition 7.15]{JoyalTierney:QCSS} it suffices to show the left adjoint preserves (generating) cofibrations and the right adjoint preserves fibrations between fibrant objects, and both of these properties are immediate from 
\cref{def sCat model structure}.
\end{proof}

We regard the ordinal categories $\catone$, $\cattwo$, $\ldots, \catn$ as discrete simplicial categories.

\begin{const}\label{defn:graph-of-pasting-scheme}
For any pasting scheme $\cP$, define  a free simplicial category $\mG\cP$
\[
\begin{tikzcd} & \aamalg{\text{edge}\cP} \catone+\catone \arrow[d]\arrow[r, hook] \arrow[dr, phantom, "\ulcorner" very near end] & \aamalg{\text{edge}\cP} \cattwo \arrow[d] \\ \varnothing \arrow[r, hook] &
\aamalg{\ob\cP} \catone \arrow[r, hook] & \mG\cP_0 \arrow[r, hook] & \mG\cP \\ &
& \aamalg{\text{face}\cP} \cattwo[\partial\Delta^1] \arrow[u] \arrow[r, hook] \arrow[ur, phantom, "\llcorner" very near end] & \aamalg{\text{face}\cP}\cattwo[\Delta^1] \arrow[u]
\end{tikzcd}
\]
 by
\begin{enumerate}[label=(\roman*)]
\item first attaching all its objects, to form $\amalg_{\ob \cP} \catone$,
\item then attaching all its atomic 1-cells along maps $\catone+\catone \cong \cattwo[\partial\Delta^0] \hookrightarrow \cattwo[\Delta^0] \cong \cattwo$,\footnote{Since this colimit is formed in simplicial categories, the result is the free category generated by the underlying 1-graph (as opposed to the underlying 1-graph itself).} and
\item finally attaching all its atomic 2-cells along maps $\cattwo[\partial\Delta^1] \to \cattwo[\Delta^1]$.
\end{enumerate}
\end{const}

\begin{lem}\label{lem:graph-computad}
For any pasting scheme $\cP$, the simplicial category $\mG{\cP}$ is a simplicial computad whose atomic $0$-arrows are the edges of $\cP$, whose atomic $1$-arrows are the atomic $2$-cells of $\cP$, and whose higher atomic arrows are all degenerate.
\end{lem}
\begin{proof}
The simplicial category $\mG{\cP}$ is built as a cell complex from the generating Bergner cofibrations and so is cofibrant, with the atomic cells freely attached by the maps $\cattwo[\partial\Delta^0] \hookrightarrow \cattwo[\Delta^0]$ and $\cattwo[\partial\Delta^1] \to \cattwo[\Delta^1]$.
\end{proof}

Alternatively we could form the free 2-category $\mF\cP$ on $\cP$ and then apply the nerve $N_* \colon \twoCat \to\sCat$ to produce a simplicial category $N_*\mF\cP$ considered in \cref{prop:pasting-computad}, equipped with natural maps
\[
\aamalg{\text{ob}\cP} \catone \to N_*\mF\cP, \quad \aamalg{\text{edge}\cP} \cattwo \to N_*\mF\cP, \quad 
\aamalg{\text{face}\cP} \cattwo[\Delta^1] \to N_*\mF\cP
\]
inducing a canonical comparison
\begin{equation}\label{eq:natural-comparison}
\mG\cP \to N_*\mF\cP.
\end{equation}
We observe that:

\begin{lem}\label{lem:natural-inclusion} The natural map $\mG\cP \rightarrowtail N_*\mF\cP$ is a simplicial subcomputad inclusion, and hence a cofibration in the model structure for $(\infty,2)$-categories.
\end{lem}
\begin{proof} 
By \cref{lem:graph-computad} and \cref{prop:pasting-computad} both simplicial categories are simplicial computads with the same objects and the same atomic 0-arrows. The atomic 1-arrows of $\mG \cP$, corresponding to atomic 2-cells of $\cP$, of course define horizontally indecomposable 2-cells of $\mF \cP$, and so are atomic 1-arrows of $N_*\mF \cP$. Thus the natural inclusion is a simplicial subcomputad inclusion, meaning $N_*\mF \cP$ may be built from $\mG \cP$ by attaching the remaining atomic $n$-arrows for $n \geq 1$ along their boundaries.
\end{proof}

Our aim is to prove that the natural map is not only a cofibration, but also a weak equivalence.

\begin{thm}[uniqueness of pasting composition]\label{thm:uniquepasting} The natural map 
\[
\begin{tikzcd}
\mG\cP \arrow[r, "\sim", tail] & N_*{\mF}\cP
\end{tikzcd}
\]
is a trivial cofibration in the model structure for $(\infty,2)$-categories.
\end{thm}

The proof of this result occupies the remainder of this section.

\subsection{Base cases and the special inductive step}\label{ssec:base-cases}

\cref{thm:uniquepasting} is easily verified for the following simple pasting schemes from \cref{ex: theta 2}:

\begin{ex}\label{ex:base-case}
When $\cP = [n]([0],\ldots,[0])$ then 
$\mG\cP \cong N_*{\mF}\cP$ is a 1-category discretely embedded as a simplicial category. 
\end{ex}

\begin{ex} When $\cP = [n]([1],\ldots,[1])$, then $\mG\cP \cong N_*{\mF}\cP$ is a simplicial category with objects $0,1,\ldots, n$ and the $k$-cube $(\Delta^1)^k$ as the hom from $i$ to $i+k$, with all other homs empty.
\end{ex}

Interpolating between these two examples, we could consider the objects of $\TTheta_2$ of the form $\cP = [n]([e_1],\ldots, [e_n])$ where each $e_i$ is 0 or 1. As in the previous examples, $\mG\cP \cong N_*{\mF}\cP$ in these cases.

\begin{ex} When $\cP= [1]([n])$, then $\mG\cP$ is the directed suspension of the spine of $\Delta^n$, while $N_*{\mF}\cP = \cattwo[\Delta^n]$. Since the inclusion of the spine is a Joyal trivial cofibration, the map $\begin{tikzcd}[cramped,sep=small]
\mG\cP \arrow[r, "\sim", tail] & N_*{\mF}\cP
\end{tikzcd}$ is a  trivial cofibration by \cref{lem:directed-suspension-adjunction}.
\end{ex}

By \cref{cor:pasting-scheme-presentations}, a generic pasting scheme can be built from these  base cases by iteratively attaching an atomic 2-cell $\alpha$ ``along the bottom,'' meaning that the domain path is identified via an attaching map $\dom_\alpha \colon \cattwo \to \mG\cP$ with a path of 1-cells $\dom_\alpha \subset \cod_{\cP}$ in $\cP$  with source object $x$ and target object $y$. We first consider the special case where the codomain of the attached cell is given by a single atomic 1-cell, rather than a path of atomic 1-cells. By this simplifying assumption $\mG(\cP\cup\alpha)$ is built from $\mG\cP$ by the pushout displayed below-left:
\begin{equation}\label{diagram:GvsNF_po}
\begin{tikzcd} \cattwo[\Delta^0] \arrow[r, "\dom_\alpha"] \arrow[d, tail, "0"'] \arrow[dr, phantom, "\ulcorner" very near end] & \mG\cP \arrow[d, tail] \arrow[r, tail, "\sim"] \arrow[dr, phantom, "\ulcorner" very near end]  & N_*\mF \cP \arrow[d, tail] \arrow[ddr, tail, bend left=10]   \\ \cattwo[\Delta^1] \arrow[r, "{(\alpha,\cod_\alpha)}"'] & \mG(\cP \cup \alpha)\arrow[r, tail, "\sim"] \arrow[drr, bend right=15, tail] & \mG(\cP \cup \alpha) \aamalg{\mG\cP} N_*\mF \cP \arrow[dr, dashed] \\
 & & & N_* \mF (\cP \cup \alpha)
\end{tikzcd}
\end{equation}
Note that all of the simplicial categories that appear here are cofibrant. Moreover, by \cref{rmk:sub-pasting-sub-computad} and direct inspection  the inclusion 
\begin{equation}\label{eq:inductive-comparison}
\mG(\cP \cup \alpha) \aamalg{\mG\cP} N_*{\mF}\cP \rightarrowtail N_*{\mF}(\cP \cup \alpha)
\end{equation} is a simplicial subcomputad inclusion, though we shall not need this fact.

Under the assumption that $\cwtoo{\mG\cP}{N_*\mF\cP}$ is a trivial cofibration, then so is its pushout, which  the map $\mG(\cP \cup \alpha) \rightarrowtail N_*{\mF}(\cP \cup \alpha)$ factors through. So to demonstrate that $\mG(\cP\cup\alpha) \rightarrowtail N_*\mF(\cP\cup\alpha)$ is a trivial cofibration, it remains to prove that the map  \eqref{eq:inductive-comparison} is a weak equivalence in the model structure for $(\infty,2)$-categories. This simplicial functor is identity on objects, so by \cref{def sCat model structure} it suffices to show that for each hom we have a weak categorical equivalence.

We use the following lemma to calculate the homs in the simplicial category $\mG(\cP \cup \alpha) \amalg_{\mG\cP} N_*{\mF}\cP$.

\begin{lem}\label{lem:phil} 
Suppose $x \prec y$ are vertices in a pasting scheme $\cP$. Given a span of simplicial sets $K \xleftarrow{f} J \xrightarrow{g}  N\mF\cP(x,y)$, the simplicial category $\mP$ defined by the pushout\footnote{Here the simplicial functor $g \colon \cattwo[J] \to N_*\mF\cP$ is the transpose of the map $g \colon J \to N\mF\cP(x,y)$ under the adjunction of \cref{lem:directed-suspension-adjunction}; note also by that result that  if $f \colon J \to K$ is a monomorphism, then $\cattwo[f] \colon \cattwo[J] \to \cattwo[K]$ is a cofibration.}
\[
\begin{tikzcd}[column sep=large] \cattwo[J] \arrow[r, "g"] \arrow[d, "{\cattwo[f]}"'] \arrow[dr, phantom, "\ulcorner" very near end] & N_*\mF{\cP} \arrow[d] \\ \cattwo[K] \arrow[r] & \mP
\end{tikzcd}
\] 
has the same objects as $N_*\mF{\cP}$ and has hom-categories defined by the pushouts 
\[
\begin{tikzcd}[column sep=small] N\mF \cP(y,z)\! \times\! J\! \times\! N\mF \cP (a,x) \arrow[r, "g"] \arrow[drr, phantom, "\ulcorner" very near end]\arrow[d, "f"'] & [-1mm] N\mF \cP(y,z)\! \times\! N\mF \cP(x,y)\! \times\! N\mF \cP (a,x) \arrow[r, "\circ"] & [-5pt] N\mF \cP(a,z) \arrow[d] \\N\mF\! \cP(y,z) \times\! K\! \times\! N\mF \cP (a,x)  \arrow[rr] & & \mP(a,z)
\end{tikzcd}
\]
for any pair of  objects $a$ and $z$.
\end{lem}

We will prove this in a moment, after establishing the following related lemma for ordinary categories.

\begin{lem}\label{lem:one-way-pushout}
Let $\cC$ be a small \textbf{one-way} category, that is, 
\begin{enumerate}[label=(\roman*)]
    \item $\cC(x,x) = \{ \id_x \}$ for every object $x$, and \label{eq:onewayid}
    \item the relation $x \peq y$, holding just when $\cC(x,y)$ is inhabited, is anti-symmetric, hence constitutes a partial order on the set of objects of $\cC$. \label{eq:onewayas}
\end{enumerate} 
Suppose $c_0 \prec c_1$ are objects of $\cC$.
 Given a span of sets $T \xleftarrow{f} S \xrightarrow{g} \cC(c_0,c_1)$, the category $\cD$ defined by the pushout
\begin{equation}\label{eq:pushout_one_way} \begin{tikzcd}
\cattwo[S] \rar{g} \dar[swap]{\cattwo[f]} \arrow[dr, phantom, "\ulcorner" very near end] & \cC \dar \\
\cattwo[T] \rar & \cD
\end{tikzcd} \end{equation}
has the same objects as $\cC$ and has hom-sets defined by the pushouts
\begin{equation}\label{eq:pushout_one_way_levels} \begin{tikzcd}[column sep=large]
\cC (c_1,b) \times S \times \cC (a,c_0) \rar{-\circ g(-) \circ -}\dar[swap]{\id \times f \times \id} \arrow[dr, phantom, "\ulcorner" very near end] & \cC (a,b)\dar \\
\cC (c_1,b) \times T \times \cC (a,c_0) \rar & \cD(a,b)
\end{tikzcd} \end{equation}
for every pair of objects $a,b$.
\end{lem}
\begin{proof}
Let $X = X' \amalg \{c_0, c_1\}$ be the set of objects of $\cC$, and let $\Cat_X$ be the category of small categories having object set $X$ and identity-on-objects functors between such.
Observe that the pushout \eqref{eq:pushout_one_way} factors as
\[ \begin{tikzcd}
\cattwo[S] \rar \dar[swap]{\cattwo[f]} \arrow[dr, phantom, "\ulcorner" very near end] & \cattwo[S]\amalg X' \rar \dar{\cattwo[f] \amalg \id}  \arrow[dr, phantom, "\ulcorner" very near end] &  \cC \dar \\
\cattwo[T] \rar & \cattwo[T]\amalg X' \rar & \cD
\end{tikzcd} \]
where both squares are pushouts in $\Cat$.
We regard the right-square as a pushout in $\Cat_X$, which can be presented as the reflexive coequalizer
\begin{equation}\label{eq:Xcat_coequalizer} \begin{tikzcd}
(\cattwo[T]\amalg X') \aamalg{X} (\cattwo[S]\amalg X')  \aamalg{X} \cC \rar[shift left] \rar[shift right] & (\cattwo[T]\amalg X')  \aamalg{X} \cC \rar & \cD.
\end{tikzcd} \end{equation}
As the monadic forgetful functor $\Cat_X \to \Set^{X\times X}$ preserves reflective coequalizers, it suffices to calculate \eqref{eq:Xcat_coequalizer} in the latter category.
Using that $\cC$ is a one-way category and $c_0 \prec c_1$, each morphism in $(\cattwo[T]\amalg X')  \amalg_{X} \cC$ from $a$ to $b$ can be written uniquely in exactly one of the following forms
\[ \begin{tikzcd}
a \rar{p} & c_0 \rar{t} & c_1 \rar{q} & b & \text{or} & a \rar{r} & b
\end{tikzcd} \]
where $p,q,r$ are morphisms of $\cC$ and $t$ is an element of $T$.
The left category of \eqref{eq:Xcat_coequalizer} additionally adds morphisms of the form 
\[ \begin{tikzcd}
a \rar{p} & c_0 \rar{s} & c_1 \rar{q} & b
\end{tikzcd} \]
where $s\in S$.
Such a morphism is taken to 
\[ \begin{tikzcd}
a \rar{p} & c_0 \rar{f(s)} & c_1 \rar{q} & b & \text{resp.} & a \rar{q\circ g(s) \circ p} &[2.5em] b
\end{tikzcd} \]
by the parallel functors, hence \eqref{eq:pushout_one_way_levels} is a pushout. 
\end{proof}

\begin{proof}[Proof of \cref{lem:phil}]
Pushouts in simplicial categories can be computed levelwise in $\Cat^{\DDelta^\op}$, and we will apply \cref{lem:one-way-pushout} in each simplicial degree.
We merely need to observe that for each simplicial level $n$, the category $(N_*\mF{\cP})_n$ is a one-way category.
This is true when $n=0$ since $(N_*\mF{\cP})_0$ is just the underlying 1-category of $\mF{\cP}$ by  \cref{prop:pasting-computad}.
This implies condition \ref{eq:onewayas} holds for all $n$, as the simplicial set $N_*\mF{\cP}(x,y)$ has vertices $(N_*\mF{\cP})_0(x,y)$, so $(N_*\mF{\cP})_0(x,y)$ being empty implies $(N_*\mF{\cP})_n(x,y)$ is empty as well.
On the other hand, $\mF{\cP}(x,x)$ is the terminal category by \cref{cor:hom-poset}. 
Hence $(N_*\mF{\cP})_n(x,x) = N(\mF{\cP}(x,x))_n$ is a one-point set for every $n$, so condition \ref{eq:onewayid} holds for each $n$. 
\end{proof}

In particular, by \eqref{diagram:GvsNF_po} and \cref{lem:phil}, the hom in $\mG(\cP \cup \alpha) \amalg_{\mG\cP} N_*{\mF}\cP$ from $x$ to $y$ is $N\mF\cP(x,y)\amalg_{\Delta^0} \Delta^1$, where the 1-simplex is glued along its domain to the path $p$. By contrast, by  \cref{lem:cell-at-the-bottom}, $\mF (\cP \cup \alpha)(x,y) \cong \mF \cP(x,y) \amalg_\catone \cattwo$, so $N\mF(\cP\cup\alpha)(x,y) \cong N(\mF\cP(x,y) \amalg_\catone\cattwo)$. Using the fact that the inclusion of the initial object $\catone\hookrightarrow\cattwo$ is a Dwyer map, by \cref{AnodyneDwyer}  it follows that the map
\[N\mF \cP(x,y)\aamalg{\Delta^0} \Delta^1 \to N (\mF \cP(x,y) \aamalg\catone\cattwo)\]
is a trivial cofibration in the Joyal model structure.
By a similar argument:

\begin{cor}\label{cor:special-inductive-step} For any pasting scheme $\cP$, suppose the natural map $\mG\cP \rightarrowtail N_*\mF\cP$ is a weak equivalence in the model structure for $(\infty,2)$-categories.
Then the natural map 
\[\mG(\cP \cup \alpha) \aamalg{\mG\cP} N_*{\mF}\cP \rightarrowtail N_*{\mF}(\cP \cup \alpha)\]
defined by attaching an atomic $2$-cell $\alpha$ with atomic codomain along a path $\dom_\alpha \subset \cod_\cP$ is a weak equivalence in the model structure for $(\infty,2)$-categories. Thus,
\[
\begin{tikzcd} \mG(\cP \cup \alpha)  \arrow[r, tail, "\sim"] & N_*{\mF}(\cP \cup \alpha)\end{tikzcd}\]
is a trivial cofibration in the model structure for $(\infty,2)$-categories.
\end{cor}
\begin{proof}
As this functor is identity-on-objects, it suffices to argue that it defines a Joyal trivial cofibration on homs for any $a,z \in \cP$. Applying \cref{lem:phil} to the pushout
\[
\begin{tikzcd} \cattwo[\Delta^0] \arrow[r, "\dom_\alpha"] \arrow[d, tail, "0"'] \arrow[dr, phantom, "\ulcorner" very near end]  &N_*\mF \cP \arrow[d, tail]    \\ \cattwo[\Delta^1] \arrow[r] & \mG(\cP \cup \alpha) \aamalg{\mG\cP} N_*\mF \cP 
\end{tikzcd}
\]
we see that the hom-spaces in $\mP \coloneqq \mG(\cP \cup \alpha) \amalg_{\mG\cP} N_*\mF \cP$ are defined by the pushouts
\[
\begin{tikzcd}[column sep=huge] N\mF \cP(y,z) \times \Delta^0 \times N\mF \cP (a,x) \arrow[dr, phantom, "\ulcorner" very near end]\arrow[d, "0"']  \arrow[r, "-\circ\dom_\alpha \circ -"] & [-5pt] N\mF \cP(a,z) \arrow[d] \\N\mF \cP(y,z) \times\Delta^1 \times N\mF \cP (a,x)  \arrow[r] &  \mP(a,z).
\end{tikzcd}
\]

Meanwhile, by \cref{lem:cell-at-the-bottom} and \cref{prop:composition-pushout} the hom-category $\mF(\cP\cup\alpha)(a,z)$ is defined by the pushout
\[
\begin{tikzcd}[column sep=large] \mF \cP(y,z) \times \catone \times \mF \cP (a,x) \arrow[dr, phantom, "\ulcorner" very near end]\arrow[d, "0"']  \arrow[r, "-\circ\dom_\alpha \circ -"] & [-5pt] \mF \cP(a,z) \arrow[d] \\\mF \cP(y,z) \times\cattwo\times \mF \cP (a,x)  \arrow[r] &  \mF(\cP\cup\alpha)(a,z).
\end{tikzcd}
\]
Thus, the map of hom simplicial sets
\[(\mG(\cP \cup \alpha) \aamalg{\mG\cP} N_*{\mF}\cP)(a,z) \to N_*{\mF}(\cP \cup \alpha)(a,z)\]
is the induced map in \cref{cor:hom-dwyer-map} and is an inner anodyne extension and in particular a Joyal trivial cofibration, as desired. 

Finally, by \eqref{diagram:GvsNF_po}, the cofibration $\mG(\cP\cup\alpha) \rightarrowtail N_*\mF(\cP\cup\alpha)$ has been expressed as a composite of weak equivalences, and thus must be one as well.
\end{proof}

\subsection{The general inductive step}\label{ssec:general-inductive-step}

We now build on our work in the previous section to prove the generalized inductive step, which we use to prove \cref{thm:uniquepasting}. 

Let $\cP'$ be a pasting scheme, and let $x \xrightarrow{e} y$ be any edge.
Define a new pasting scheme $\cP$ by replacing $e$ by a path of edges
\[ x \xrightarrow{e_1} v_1 \xrightarrow{e_2} v_2 \to \cdots\to v_{n-2}\xrightarrow{e_{n-1}} y,\]
without changing any of the other data of the pasting scheme, such as the image of the planar embedding or the names of the atomic 2-cells.

This modification has a relatively modest effect on the free 2-categories generated by these pasting schemes, which may be summarized by observing that 
\[ \begin{tikzcd}
\cattwo \rar{e} \arrow[d, "\circ"'] \arrow[dr, phantom, "\ulcorner" very near end] &
    \mF\cP'
     \dar \\
\catn \rar[swap]{e_1\cdots e_{n-1}} & 
    \mF\cP
\end{tikzcd} \]
is a pushout in $\twoCat$, where $\circ \colon \cattwo \to \catn$ is the unique ``active'' map, which classifies the maximal composite in $\catn$. Observe also that the right vertical map is fully faithful, and if $u$ is a vertex in $\cP'$ then we have isomorphisms of hom-categories $ \mF \cP'(u,x) \cong \mF \cP(u,v_i)$ and $\mF \cP'(y,u) \cong \mF \cP(v_i,u)$ implemented by post- and pre-whiskering with the paths of edges $e_1\cdots e_i \colon x \to v_i$ and $e_{i+1}\cdots e_{n-1} \colon v_i \to y$.

By \cref{prop:pasting-computad} we may extend these observations to show:

\begin{lem}
The  diagram
\[ \begin{tikzcd}
\cattwo \rar{e} \arrow[d, "\circ"'] \arrow[dr, phantom, "\ulcorner" very near end] &
    N_*\mF\cP' \dar \\
\catn \rar[swap]{e_1\cdots e_{n-1}} & 
    N_*\mF\cP
\end{tikzcd} \]
is a pushout in $\sCat$.
\end{lem}
\begin{proof}
As in the proof of \cref{lem:phil}, we can check this by making sure it is true for the categories in each simplicial degree $k$. In degree $k$, the pushout is formed by factoring the atomic $k$-arrow $e$ as $e_1\cdots e_{n-1}$. Since the categories of $k$-arrows are one-way, the effect of this factorization on hom-sets is easy to describe. By \cref{prop:pasting-computad}, the isomorphisms of hom-categories just described induce isomorphisms
\begin{align*}
N(\mF\cP (a,z))_k &\cong N(\mF\cP'(a,z))_k & 
    N(\mF\cP (a,v_i))_k &\cong N(\mF\cP'(a,x))_k \\
N(\mF\cP (v_i,v_j))_k &\cong \{e_{i+1}\cdots e_j\} & 
    N(\mF\cP (v_i,z))_k &\cong N(\mF\cP'(y,z))_k \\
N(\mF\cP (v_j,v_i))_k &= \varnothing & 
    N(\mF\cP (v_i,v_i))_k &\cong \{v_i\}
\end{align*}
where $a$ and $z$ are vertices of $\cP'$ and $1\leq i < j \leq n-2$. The claim follows.
\end{proof}

Consulting \cref{defn:graph-of-pasting-scheme}, we find another pushout of simplicial categories.
\begin{lem}
The diagram
\[ \begin{tikzcd}
\cattwo \rar{e} \arrow[d, "\circ"'] \arrow[dr, phantom, "\ulcorner" very near end] &
    \mG\cP' \dar \\
\catn \rar[swap]{e_1\cdots e_{n-1}} & 
    \mG\cP
\end{tikzcd} \]
is a pushout in $\sCat$.
\end{lem}
\begin{proof}
The first layer of the construction of \cref{defn:graph-of-pasting-scheme} builds the free category $\mG\cP_0$ associated to the underlying graph of the pasting scheme $\cP$. This recovers the underlying 1-category of the 2-category $\mF\cP$. Thus, by passing to underlying 1-categories, either of the pushouts above give rise to a pushout
\[ \begin{tikzcd}
\cattwo \rar{e} \arrow[d, "\circ"'] \arrow[dr, phantom, "\ulcorner" very near end] &
    \mG\cP'_0 \dar \\
\catn \rar[swap]{e_1\cdots e_{n-1}} & 
    \mG\cP_0.
\end{tikzcd} \]

As $\cP'$ and $\cP$ have the same atomic 2-cells, we have the following diagram
\[ \begin{tikzcd}
& \coprod \cattwo[\partial \Delta^1] \arrow[r] \dar & \coprod \cattwo[\Delta^1] \dar \\
\cattwo \rar \dar["\circ"'] \arrow[dr, phantom, "\ulcorner" very near end] &  \mG \cP'_0  \rar \dar  & \mG \cP' \dar \\
\catn \rar &   \mG\cP_0  \rar  & \mG\cP.
\end{tikzcd} \]
We just argued that the left square is a pushout. 
The top square and the right rectangle are pushouts by definition of $\mG\cP'$ and $\mG\cP$, respectively.
Hence the bottom right square is a pushout, which implies that the bottom rectangle is a pushout as well.
\end{proof}

By combining the previous two lemmas, we conclude:
\begin{cor}\label{cor:inductive-pushout}
The natural inclusions assemble into a pushout square in $\sCat$:
\[ \begin{tikzcd}
\mG \cP' \dar \arrow[r, tail] \arrow[dr, phantom, "\ulcorner" very near end] & N_*\mF\cP' \dar \\
\mG \cP \arrow[r, tail] & N_*\mF\cP.
\end{tikzcd} \]
\end{cor}
\begin{proof}
The outer rectangle and left square of the following are pushouts, and thus so is the right-hand square:
\[ \begin{tikzcd}
\cattwo \rar{e} \arrow[d, "\circ"']  \arrow[dr, phantom, "\ulcorner" very near end] &
    \mG \cP' \dar \arrow[r, tail]  & N_*\mF\cP' \dar \\
\catn \rar[swap]{e_1\cdots e_{n-1}} & \mG \cP \rar &
    N_*\mF\cP.
\end{tikzcd} \]
\par \vspace{-1.3\baselineskip} \qedhere
\end{proof}

Note the vertical functors in \cref{cor:inductive-pushout} are \emph{not} cofibrations in the model structure for $(\infty,2)$-categories: because $\cP'$ does not define a sub pasting scheme of $\cP$, the simplicial category $N_*\mF\cP'$ is not a simplicial subcomputad of $N_*\mF\cP$. Nevertheless we can apply this result to deduce the general inductive step from the special case of  \cref{cor:special-inductive-step}.

\begin{cor}\label{cor:general-inductive-step} For any pasting scheme $\cP$, suppose the natural map $\mG\cP \rightarrowtail N_*\mF\cP$ is a weak equivalence in the model structure for $(\infty,2)$-categories.
Then the natural map 
\[
\begin{tikzcd} \mG(\cP \cup \alpha)  \arrow[r, tail, "\sim"] & N_*{\mF}(\cP \cup \alpha)\end{tikzcd}\]
defined by attaching an atomic $2$-cell $\alpha$ along a path $\dom_\alpha \subset \cod_\cP$ is a trivial cofibration in the model structure for $(\infty,2)$-categories.
\end{cor}
\begin{proof} If the path $\cod_\alpha$ consists of a single edge, this result was proven in \cref{cor:special-inductive-step}. If instead $\cod_\alpha = e_1\cdots e_{n-1}$ then we consider instead an atomic 2-cell $\alpha'$ with $\dom_{\alpha'} = \dom_\alpha$ and $\cod_\alpha =e$ given by a single edge and form the pasting scheme $\cP \cup \alpha'$. By \cref{cor:inductive-pushout}, we have a pushout of simplicial categories
\[ \begin{tikzcd}
\mG (\cP \cup \alpha') \dar \arrow[r, tail, "\sim"]  \arrow[dr, phantom, "\ulcorner" very near end] & N_*\mF(\cP \cup \alpha') \dar \\
\mG (\cP \cup \alpha) \arrow[r, tail] & N_*\mF(\cP \cup \alpha)
\end{tikzcd} \]
and by \cref{cor:special-inductive-step} the top horizontal functor is a trivial cofibration in the model structure for $(\infty,2)$-categories. Thus the bottom horizontal functor must be as well.
\end{proof}

We conclude by assembling these results into a proof of our main theorem.

\begin{proof}[Proof of \cref{thm:uniquepasting}]
By \cref{cor:pasting-scheme-presentations} any pasting scheme $\cP$ can be built inductively from the base case of \cref{ex:base-case} by iteratively attaching the atomic 2-cells along a subpath of the codomain path. We prove that $\mG\cP \rightarrowtail N_*\mF\cP$ is a trivial cofibration in the model structure for $(\infty,2)$-categories by induction over the number of atomic 2-cells in $\cP$.

In the base case, discussed in \cref{ex:base-case}, this map in fact defines an isomorphism of simplicial categories. For the inductive step, we use \cref{prop:top-face} to identify an atomic 2-cell $\alpha$ in $\cP$ with $\cod_\alpha \subset \cod_{\cP}$. By \cref{cor:delete-inner-face}, $\cP\backslash \alpha$ is a sub pasting scheme with one fewer atomic $2$-cell, so by our inductive hypothesis the natural map $\mG(\cP\backslash \alpha) \rightarrowtail N_*\mF(\cP\backslash \alpha)$ is a trivial cofibration in the model structure for $(\infty,2)$-categories. By \cref{cor:general-inductive-step}, it follows that $\cwtoo{\mG \cP}{N_*\mF\cP}$ is a trivial cofibration in the model structure for $(\infty,2)$-categories as desired.
\end{proof}

\subsection{Model independence}\label{ssec:model-independence}

Suppose that $\Cat_{(\infty,2)}$ is a theory of $(\infty,2)$-categories in the sense of Barwick and Schommer-Pries. This means that $\Cat_{(\infty,2)}$ is an $(\infty,1)$-cat\-e\-gory equipped with a nerve functor $N \colon \Gaunt_2 \to \Cat_{(\infty,2)}$ defined on the full subcategory $\Gaunt_2$ of $\twoCat$ consisting of compact gaunt 2-categories, which satisfies the axioms (C.1)--(C.5) of \cite[\textsection7]{BSP}. A 2-category is \textbf{gaunt} just when all of its 1-cell and 2-cell isomorphisms are in fact identities, meaning in particular that all 1-cell equivalences are in fact identities. In particular, the 2-category $\mF\cP$ associated to any pasting scheme $\cP$ is gaunt.

An example of such a theory of $(\infty,2)$-categories is the underlying $(\infty,1)$-category of $\sCat$ above, when equipped with the restriction of the hom-wise nerve functor $N_*$; details about the identification of this nerve appear in \cite{MOR:nerves}. 
Other examples include Barwick's $2$-fold complete Segal spaces \cite{BarwickThesis}, Verity's $2$-complicial sets \cite{Verity:2008wcI,OR:modelstructure}, Lurie's $\infty$-bicategories \cite{LurieGoodwillie,GHL}, Rezk's complete Segal $\TTheta_2$-spaces \cite{rezkTheta}, Bergner--Rezk's Segal categories in complete Segal spaces \cite{Bergner-Rezk-I}, and Ara's $2$-quasi-categories \cite{Ara}.

The $(\infty,1)$-category $\Cat_{(\infty,2)}$ is presentable, meaning in particular that it has all colimits. This allows us to define the \textbf{graph} $\mG\cP \in \Cat_{(\infty,2)}$ of a pasting scheme $\cP$ by the homotopy colimit
\[
\begin{tikzcd} & \aamalg{\text{edge}\cP} N\partial C_1 \arrow[d]\arrow[r] \arrow[dr, phantom, "\ulcorner" very near end] & \aamalg{\text{edge}\cP} NC_1 \arrow[d] \\ \varnothing \arrow[r] &
\aamalg{\ob\cP} NC_0 \arrow[r] & \mG\cP_0 \arrow[r] & \mG\cP \\ &
& \aamalg{\text{face}\cP} N\partial C_2 \arrow[u] \arrow[r] \arrow[ur, phantom, "\llcorner" very near end] & \aamalg{\text{face}\cP} NC_2 \arrow[u]
\end{tikzcd}
\]
built from the 0-, 1-, and 2-cells $NC_0, NC_1, NC_2 \in \Cat_{(\infty,2)}$. The canonical maps
\[
\aamalg{\text{ob}\cP} NC_0 \to N\mF\cP, \quad \aamalg{\text{edge}\cP} NC_1 \to N\mF\cP, \quad 
\aamalg{\text{face}\cP} NC_2 \to N\mF\cP\]
induce a well-defined comparison map $\mG\cP \to N\mF\cP$ in $\Cat_{(\infty,2)}$.

\begin{thm}\label{thm:model-independent-comparison}
For any pasting scheme $\cP$, the canonical map $\mG \cP \to N\mF\cP$ is an equivalence in $\Cat_{(\infty,2)}$.
\end{thm}
\begin{proof}
Our earlier construction of $\mG\cP$ in $\sCat$ from \cref{defn:graph-of-pasting-scheme} is in fact a homotopy colimit by \cite[Proposition A.2.4.4]{HTT}, hence agrees with the construction of $\mG\cP$ in the underlying $(\infty,1)$-category of $\sCat$. This means that in the model of the theory of $(\infty,2)$-categories presented by \cref{def sCat model structure}, the canonical map $\mG\cP\to N\mF\cP$ coincides with the one defined by \eqref{eq:natural-comparison}. We showed in \cref{thm:uniquepasting} that this is an equivalence. 
 
 By \cite{BSP}, any change of models of the theory of $(\infty,2)$-categories defines an equivalence of $(\infty,1)$-categories that is compatible with the nerve.  Since equivalences of $(\infty,1)$-categories preserve colimits, the map in question is well-defined up to change of models of the theory of $(\infty,2)$-categories. Since equivalences in an $(\infty,1)$-category are preserved and reflected by equivalences between $(\infty,1)$-cat\-e\-gories, the result follows.
\end{proof}

As a direct corollary, it follows that any pasting diagram in any $(\infty,2)$-category has a homotopically unique composite. The $(\infty,2)$-category of composites of a pasting diagram $d \colon \mG\cP \to \cC$ in an $(\infty,2)$-category is defined to be the pullback
\[
\begin{tikzcd} \bullet \arrow[r, dashed] \arrow[d, dashed, "\rotatebox{90}{$\sim$}"'] \arrow[dr, phantom, "\lrcorner" very near start] & \cC^{N\mF\cP} \arrow[d, "\rotatebox{90}{$\sim$}"] \\ NC_0 \arrow[r, "d"] & \cC^{\mG\cP}
\end{tikzcd}
\]
in $\Cat_{(\infty,2)}$, where the exponentials are the internal homs guaranteed by axiom (C3) of \cite{BSP}.  By \cref{thm:model-independent-comparison} the map $\mG\cP \wto N\mF\cP$ is an equivalence, hence so is the restriction functor, which pulls back to define an equivalence between the $(\infty,2)$-category of composites and the terminal $(\infty,2)$-category. Hence:

\begin{cor}\label{cor:contractible-composites}
The space of composites of any pasting diagram in any $(\infty,2)$-cat\-e\-gory is contractible. \qed
\end{cor}
 
 In particular, this $(\infty,2)$-category is an $\infty$-groupoid, justifying our calling it the ``space'' of composites.

\begin{rmk}\label{rmk:contractible-composites} For $\cB,\cC \in \Cat_{(\infty,2)}$, the defining universal property of $\cC^{\cB}$ characterizes this $(\infty,2)$-category as the right adjoint to the cartesian product. Thus we interpret its 0-cells as pseudofunctors $\cB \to \cC$, its 1-cells as pseudonatural transformations between these, and its 2-cells as ``modifications.'' In certain models of $(\infty,2)$-categories, such as 2-quasi-categories,  scaled simplicial sets, or 2-complicial sets, it might be preferable to work with an alternative internal hom $\text{Fun}_{\text{oplax}}(\cB,\cC)$ defined as the right adjoint to the Gray tensor product; see \cite{Maehara-gray}, \cite{GHL-gray}, or \cite{Verity:2008cs, ORV} for details. The 0-cells of $\text{Fun}_{\text{oplax}}(\cB,\cC)$ are again pseudofunctors $\cB \to \cC$, but now the 1-cells are oplax natural transformations. The equivalence $\mG\cP \wto N\mF\cP$ again induces an equivalence $\text{Fun}_{\text{oplax}}(N\mF\cP,\cC) \to \text{Fun}_{\text{oplax}}(\mG\cP,\cC)$, so the fibers of this map are contractible $\infty$-groupoids. 

An $(\infty,2)$-category $\text{Fun}_{\text{icon}}(\cB,\cC)$ of pseudofunctors and ``icons''---an appellation given by analogy to the identity component oplax natural transformations in 2-category theory defined by Lack \cite{Lack:Icons}---can be defined to be the sub $(\infty,2)$-cat\-e\-gory spanned by those oplax natural transformations whose components are equivalences. Since the functor $\mG\cP \wto N\mF\cP$ is surjective on 0-cells, the square 
\[
\begin{tikzcd} \text{Fun}_{\text{icon}}(N\mF\cP,\cC) \arrow[r, hook] \arrow[d,  "\rotatebox{90}{$\sim$}"'] \arrow[dr, phantom, "\lrcorner" very near start] & \text{Fun}_{\text{oplax}}(N\mF\cP,\cC)\arrow[d, "\rotatebox{90}{$\sim$}"] \\ \text{Fun}_{\text{icon}}(\mG\cP,\cC) \arrow[r,hook] & \text{Fun}_{\text{oplax}}(\mG\cP,\cC)
\end{tikzcd}
\]
is a pullback, so the fibers of the map between the sub $(\infty,2)$-categories coincide with the fibers of the map between the larger $(\infty,2)$-categories. Morally, this should recover the $(\infty,2)$-category of icons considered by \cite[Theorem D]{Columbus:2-cat} following a similar construction in \cite[\S4.4]{RV:adj}, though we don't know a direct comparison between the constructions given in the different models. 
\end{rmk}
 
\bibliographystyle{alpha}
\bibliography{refs_inf-2}

\newcommand{\etalchar}[1]{$^{#1}$}
\begin{thebibliography}{BMO{\etalchar{+}}15}

\bibitem[AM14]{AraMaltsiniotisVers}
Dimitri Ara and Georges Maltsiniotis.
\newblock Vers une structure de cat\'{e}gorie de mod\`eles \`a la {T}homason
  sur la cat\'{e}gorie des {$n$}-cat\'{e}gories strictes.
\newblock {\em Adv. Math.}, 259:557--654, 2014.

\bibitem[Ara14]{Ara}
Dimitri Ara.
\newblock Higher quasi-categories vs higher {R}ezk spaces.
\newblock {\em J. K-Theory}, 14(3):701--749, 2014.

\bibitem[Bar05]{BarwickThesis}
Clark Barwick.
\newblock {\em $(\infty, n)$-{C}at as a closed model category}.
\newblock PhD thesis, University of Pennsylvania, 2005.

\bibitem[B{\'{e}}n67]{Benabou-Bicategories}
Jean B{\'{e}}nabou.
\newblock Introduction to bicategories.
\newblock In {\em Reports of the Midwest Category Seminar}, volume~47 of {\em
  Lecture Notes in Math.}, pages 1--77. Springer, Berlin, Heidelberg, 1967.

\bibitem[Ber07a]{Berger-Wreath}
Clemens Berger.
\newblock Iterated wreath product of the simplex category and iterated loop
  spaces.
\newblock {\em Adv. Math.}, 213(1):230--270, 2007.

\bibitem[Ber07b]{Bergner:MS}
Julia~E. Bergner.
\newblock A model category structure on the category of simplicial categories.
\newblock {\em Trans. Amer. Math. Soc.}, 359(5):2043--2058, 2007.

\bibitem[BM08]{BondyMurty}
J.~A. Bondy and U.~S.~R. Murty.
\newblock {\em Graph theory}, volume 244 of {\em Graduate Texts in
  Mathematics}.
\newblock Springer, New York, 2008.

\bibitem[BMO{\etalchar{+}}15]{BMOOPY}
Anna~Marie Bohmann, Kristen Mazur, Ang\'{e}lica~M. Osorno, Viktoriya Ozornova,
  Kate Ponto, and Carolyn Yarnall.
\newblock A model structure on {$G\mathcal Cat$}.
\newblock In {\em Women in topology: collaborations in homotopy theory}, volume
  641 of {\em Contemp. Math.}, pages 123--134. Amer. Math. Soc., Providence,
  RI, 2015.

\bibitem[BR13]{Bergner-Rezk-I}
Julia~E. Bergner and Charles Rezk.
\newblock Comparison of models for {$(\infty,n)$}-categories, {I}.
\newblock {\em Geom. Topol.}, 17(4):2163--2202, 2013.

\bibitem[BSP21]{BSP}
Clark Barwick and Christopher Schommer-Pries.
\newblock On the unicity of the homotopy theory of higher categories.
\newblock {\em J. Amer. Math. Soc.}, 34(4):1011--1058, 2021.

\bibitem[Cis99]{CisinskiDwyer}
Denis-Charles Cisinski.
\newblock La classe des morphismes de {D}wyer n'est pas stable par retractes.
\newblock {\em Cahiers Topologie G\'{e}om. Diff\'{e}rentielle Cat\'{e}g.},
  40(3):227--231, 1999.

\bibitem[Col17]{Columbus:2-cat}
Tobias Columbus.
\newblock {\em $2$-Categorical Aspects of Quasi-Categories}.
\newblock PhD thesis, Karlsruher Institut f\"{u}r Technologie, 2017.

\bibitem[FL81]{FritschLatch:HomotopyInversesForNerve}
Rudolf Fritsch and Dana~May Latch.
\newblock Homotopy inverses for nerve.
\newblock {\em Math. Z.}, 177(2):147--179, 1981.

\bibitem[For22]{Forest:Unifying}
Simon Forest.
\newblock Unifying notions of pasting diagrams.
\newblock {\em High. Struct.}, 6(1):1--79, 2022.

\bibitem[GHL21]{GHL-gray}
Andrea Gagna, Yonatan Harpaz, and Edoardo Lanari.
\newblock Gray tensor products and lax functors of {$(\infty,2)$}-categories.
\newblock {\em Adv. Math.}, 391:Paper No. 107986, 32, 2021.

\bibitem[GHL22]{GHL}
Andrea Gagna, Yonatan Harpaz, and Edoardo Lanari.
\newblock On the equivalence of all models for $(\infty,2)$-categories.
\newblock {\em J. Lond. Math. Soc. (2)}, 2022.
\newblock \href{https://arxiv.org/abs/1911.01905v2}{arXiv:1911.01905v2}
  [math.AT], \href{https://doi.org/10.1112/jlms.12614}{doi:10.1112/jlms.12614}.

\bibitem[Had20]{Hadzihasanovic}
Amar Hadzihasanovic.
\newblock A combinatorial-topological shape category for polygraphs.
\newblock {\em Appl. Categ. Structures}, 28(3):419--476, 2020.

\bibitem[HORR22]{HORR-Dwyer}
Philip Hackney, Viktoriya Ozornova, Emily Riehl, and Martina Rovelli.
\newblock Pushouts of {D}wyer maps are $(\infty,1)$-categorical.
\newblock \href{https://arxiv.org/abs/2205.02353}{arXiv:2205.02353} [math.AT],
  2022.

\bibitem[Joh89]{Johnson89}
Michael Johnson.
\newblock The combinatorics of {$n$}-categorical pasting.
\newblock {\em J. Pure Appl. Algebra}, 62(3):211--225, 1989.

\bibitem[JT07]{JoyalTierney:QCSS}
Andr\'{e} Joyal and Myles Tierney.
\newblock Quasi-categories vs {S}egal spaces.
\newblock In {\em Categories in algebra, geometry and mathematical physics},
  volume 431 of {\em Contemp. Math.}, pages 277--326. Amer. Math. Soc.,
  Providence, RI, 2007.

\bibitem[JY19]{JY}
Niles Johnson and Donald Yau.
\newblock A bicategorical pasting theorem.
\newblock \href{https://arxiv.org/abs/1910.01220v1}{arXiv:1910.01220v1}
  [math.CT], 2019.

\bibitem[JY21]{JYbook}
Niles Johnson and Donald Yau.
\newblock {\em $2$-dimensional categories}.
\newblock Oxford University Press, Oxford, 2021.

\bibitem[KS74]{KS}
G.~M. Kelly and Ross Street.
\newblock Review of the elements of {$2$}-categories.
\newblock In {\em Category {S}eminar ({P}roc. {S}em., {S}ydney, 1972/1973)},
  pages 75--103. Lecture Notes in Math., Vol. 420, 1974.

\bibitem[Lac10]{Lack:Icons}
Stephen Lack.
\newblock Icons.
\newblock {\em Appl. Categ. Structures}, 18(3):289--307, 2010.

\bibitem[Lur09a]{HTT}
Jacob Lurie.
\newblock {\em Higher topos theory}, volume 170 of {\em Annals of Mathematics
  Studies}.
\newblock Princeton University Press, Princeton, NJ, 2009.

\bibitem[Lur09b]{LurieGoodwillie}
Jacob Lurie.
\newblock {$(\infty, 2)$-categories and the Goodwillie Calculus I}.
\newblock \href{https://arxiv.org/abs/0905.0462v2}{arXiv:0905.0462v2}
  [math.CT], 2009.

\bibitem[Mae21]{Maehara-gray}
Yuki Maehara.
\newblock The gray tensor product for 2-quasi-categories.
\newblock {\em Advances in Mathematics}, 377:1--78, 2021.

\bibitem[MOR22]{MOR:nerves}
Lyne Moser, Viktoriya Ozornova, and Martina Rovelli.
\newblock Model independence of $(\infty,2)$-categorical nerves.
\newblock \href{https://arxiv.org/abs/2206.00660}{arXiv:2206.00660} [math.AT],
  2022.

\bibitem[OR20]{OR:modelstructure}
Viktoriya Ozornova and Martina Rovelli.
\newblock Model structures for ({$\infty,n$})-categories on (pre)stratified
  simplicial sets and prestratified simplicial spaces.
\newblock {\em Algebr. Geom. Topol.}, 20(3):1543--1600, 2020.

\bibitem[ORV20]{ORV}
Viktoriya Ozornova, Martina Rovelli, and Dominic Verity.
\newblock Gray tensor product and saturated $n$-complicial sets.
\newblock \href{https://arxiv.org/abs/2007.01235v1}{arXiv:2007.01235v1}
  [math.AT], 2020.

\bibitem[Pow90]{PowerPasting}
A.~J. Power.
\newblock A {$2$}-categorical pasting theorem.
\newblock {\em J. Algebra}, 129(2):439--445, 1990.

\bibitem[Rap10]{Raptis:HTP}
George Raptis.
\newblock Homotopy theory of posets.
\newblock {\em Homology Homotopy Appl.}, 12(2):211--230, 2010.

\bibitem[Rez10]{rezkTheta}
Charles Rezk.
\newblock A {C}artesian presentation of weak {$n$}-categories.
\newblock {\em Geom. Topol.}, 14(1):521--571, 2010.

\bibitem[Rie14]{Riehl:2014ch}
Emily Riehl.
\newblock {\em Categorical Homotopy Theory}, volume~24 of {\em New Mathematical
  Monographs}.
\newblock Cambridge University Press, 2014.

\bibitem[RV16]{RV:adj}
Emily Riehl and Dominic Verity.
\newblock Homotopy coherent adjunctions and the formal theory of monads.
\newblock {\em Adv. Math.}, 286:802--888, 2016.

\bibitem[Sch19]{SchwedeOrbispaces}
Stefan Schwede.
\newblock Categories and orbispaces.
\newblock {\em Algebr. Geom. Topol.}, 19(6):3171--3215, 2019.

\bibitem[Ste93]{Steiner93}
Richard Steiner.
\newblock The algebra of directed complexes.
\newblock {\em Appl. Categ. Structures}, 1(3):247--284, 1993.

\bibitem[Ste04]{SteinerEmbedding}
Richard Steiner.
\newblock Omega-categories and chain complexes.
\newblock {\em Homology Homotopy Appl.}, 6(1):175--200, 2004.

\bibitem[Str76]{Street:Limits}
Ross Street.
\newblock Limits indexed by category-valued {$2$}-functors.
\newblock {\em J. Pure Appl. Algebra}, 8(2):149--181, 1976.

\bibitem[Str91]{Street:Parity}
Ross Street.
\newblock Parity complexes.
\newblock {\em Cahiers Topologie G\'{e}om. Diff\'{e}rentielle Cat\'{e}g.},
  32(4):315--343, 1991.

\bibitem[Tho80]{ThomasonModelCat}
R.~W. Thomason.
\newblock Cat as a closed model category.
\newblock {\em Cahiers Topologie G\'{e}om. Diff\'{e}rentielle}, 21(3):305--324,
  1980.

\bibitem[Ver92]{Verity-Thesis}
Dominic Verity.
\newblock {\em Enriched categories, internal categories and change of base}.
\newblock PhD thesis, Cambridge University, 1992.

\bibitem[Ver08a]{Verity:2008cs}
Dominic Verity.
\newblock Complicial sets characterising the simplicial nerves of strict
  {$\omega$}-categories.
\newblock {\em Mem. Amer. Math. Soc.}, 193(905):xvi+184, 2008.

\bibitem[Ver08b]{Verity:2008wcI}
Dominic Verity.
\newblock Weak complicial sets. {I}. {B}asic homotopy theory.
\newblock {\em Adv. Math.}, 219(4):1081--1149, 2008.

\end{thebibliography}

\end{document}